\documentclass[11pt,reqno]{amsart}

\usepackage{mathrsfs}

\usepackage{hyperref}

\usepackage{graphicx}

\usepackage{xcolor}

\usepackage{enumerate}
\usepackage{bm}

\newtheorem{theorem}{Theorem}[section]
\newtheorem{lemma}[theorem]{Lemma}
\newtheorem{proposition}[theorem]{Proposition}
\newtheorem{corollary}[theorem]{Corollary}

\theoremstyle{definition}

\theoremstyle{remark}
\newtheorem{remark}[theorem]{Remark}

\numberwithin{equation}{section}

\begin{document}

\title[The boundary case for the supercritical deformed HYM equation]{
The boundary case for the supercritical deformed Hermitian-Yang-Mills equation 
}

\author{Wei Sun}


\address{Institute of Mathematical Sciences, ShanghaiTech University, Shanghai, China}
\email{sunwei@shanghaitech.edu.cn}



\begin{abstract}

In this paper, we shall study the weak solution to the supercritical deformed Hermitian-Yang-Mills equation in the boundary case. 

\end{abstract}

\maketitle

\medskip
\section{Introduction}

Let $(M,\omega)$ be a compact K\"ahler manifold of complex dimension $n \geq 3$.
In this paper, we shall study the {\em supercritical phase} case of deformed Hermitian-Yang-Mills equation, 
\begin{equation}
\label{equation-1-1}
\mathfrak{Re} \left(\chi + \sqrt{-1} \partial\bar\partial \varphi + \sqrt{-1} \omega\right)^n 
= 
\cot (\theta_0) \mathfrak{Im} \left(\chi + \sqrt{-1} \partial\bar\partial \varphi + \sqrt{-1} \omega\right)^n ,
\quad
\sup_M \varphi = 0,
\end{equation}
where $\theta_0 \in (0,\pi)$ and
\begin{equation*}
\mathfrak{Re} \int_M\left(\chi  + \sqrt{-1} \omega\right)^n 
= 
\cot (\theta_0) \mathfrak{Im} \int_M\left(\chi  + \sqrt{-1} \omega\right)^n .
\end{equation*}
Jacob and Yau~\cite{JacobYau2017} introduced the study on the solvability of Equation~\eqref{equation-1-1}, which is important in mirror symmetry~\cite{HarveyLawson1982} and mathematical physics~\cite{LeungYauZaslow2001}\cite{MarinoMinasianMooreStrominger2000}\cite{StromingerYauZaslow1996}.
Let $\bm{\lambda} (\chi + \sqrt{-1} \partial\bar\partial \varphi)$ denote the eigenvalue set of $\chi + \sqrt{-1} \partial\bar\partial \varphi$ with respect to $\omega$. 
Equation~\eqref{equation-1-1} can be written as
\begin{equation*}
	\dfrac{\mathfrak{Re} \left(\prod^n_{i = 1} (\lambda_i (\chi + \sqrt{-1}\partial\bar\partial\varphi) + \sqrt{-1}) \right)}{\mathfrak{Im} \left(\prod^n_{i = 1} (\lambda_i (\chi + \sqrt{-1}\partial\bar\partial\varphi)  + \sqrt{-1}) \right)} = \cot (\theta_0) .
\end{equation*}
As shown by Jacob and Yau~\cite{JacobYau2017},  
the supercritical phase case implies that Equation~\eqref{equation-1-1} can be further rewritten as
\begin{equation}
\label{equation-1-2}
	\sum^n_{i = 1} \textup{ arccot }
	\lambda_i \left(\chi + \sqrt{-1} \partial\bar\partial\varphi\right) 
	=
	\theta_0
	.
\end{equation}
In particular, Equation~\eqref{equation-1-2} is called {\em hypercritical} if $\theta_0 \in \left(0, \dfrac{\pi}{2}\right)$.
Collins, Jacob and Yau~\cite{CollinsJacobYau2020} adapted $\mathcal{C}$-subsolution~\cite{Guan2014}\cite{Szekelyhidi2018} to solve Equation~\eqref{equation-1-1}. 
They showed that
a real valued $C^2$ function is a $\mathcal{C}$-subsolution to Equation~\eqref{equation-1-1} 
if and only if 
at each point $\bm{z} \in M$,
\begin{equation}
\label{inequality-2-2}
	\sum_{i \neq j} \textup{ arccot} \left(\lambda_i (\chi + \sqrt{-1} \partial\bar\partial v) \right) < \theta_0 , \qquad \forall j = 1 , 2 , \cdots , n .  
\end{equation}
In this paper, we are concerned with the boundary case, 
%
%
\begin{equation}
\label{boundary-case-2-5}
\sum_{i \neq j} \textup{ arccot} \left(\lambda_i (\chi + \sqrt{-1} \partial\bar\partial v) \right) \leq \theta_0 ,\qquad \forall j = 1 , 2 ,\cdots , n .
\end{equation}

In dimension $2$,  we can rewrite Equation~\eqref{equation-1-1} as
\begin{equation}
\label{equation-2-9}
\left(\chi + \sqrt{-1} \partial\bar\partial \varphi - \cot(\theta_0) \omega\right)^2 
= 
\csc^2(\theta_0) \omega^2,
\end{equation}
when Condition~\eqref{boundary-case-2-5} occurs.
It is easy to see that Equation~\eqref{equation-2-9} is a complex Monge-Amp\`ere equation with semipositive and big metric $\chi + \sqrt{-1} \partial\bar\partial v - \cot (\theta_0) \omega$. 
It is known that Equation~\eqref{equation-2-9} has a unique solution in pluripotential sense~\cite{EyssidieuxGuedjZeiahi2009}.
To study the equation in higher dimensions, we need to impose some extra structure conditions. 
Indeed, we shall study the following equation
\begin{equation}
\label{equation-1-2-1}
\mathfrak{Re} \left(\chi + \tilde \chi + \sqrt{-1} \partial\bar\partial \varphi + \sqrt{-1} \omega\right)^n 
= 
\cot (\theta_0) \mathfrak{Im} \left(\chi + \tilde \chi + \sqrt{-1} \partial\bar\partial \varphi + \sqrt{-1} \omega\right)^n ,
\end{equation}
where $\sup_M \varphi = 0$, $[\chi]$ satisfies the boundary case condition~\eqref{boundary-case-2-5}, $[\tilde \chi]$ is nef and big, and
\begin{equation}
\mathfrak{Re} \int_M\left(\chi + \tilde \chi + \sqrt{-1} \omega\right)^n 
= 
\cot (\theta_0) \mathfrak{Im} \int_M\left(\chi  + \tilde \chi + \sqrt{-1} \omega\right)^n .
\end{equation}
In dimension $2$, $\chi - \cot (\theta_0) \omega$ is a natural choice for $\tilde \chi$. 
From the previous works~\cite{Yau1978}\cite{EyssidieuxGuedjZeiahi2009}\cite{FuYauZhang20201}, we know that the results in this paper still hold true when $n = 2$. 
However, we shall concentrate our research on the cases of $n \geq 3$ in this paper. 
For more details of dimensional $2$ case, we refer the reader to Fu-Yau-Zhang~\cite{FuYauZhang20201}. 

The solution to Equation~\eqref{equation-1-2-1} is probably in some weak sense. 
A classical strategy to discover a weak solution is to construct and then investigate an approximation equation.
We may choose a constant $\Theta_0 \in (\theta_0 , \pi)$, and assume that
$\tilde \chi + \omega > 0$
and
\begin{equation*}
\bm{\lambda} (\chi + \tilde \chi + \omega ) \in \Gamma_{\theta_0,\Theta_0} 
:= 
\left\{\bm{\lambda} \in \mathbb{R}^n \Bigg| \max \bigg\{\sum_{i \neq j} \textup{ arccot } \lambda_i\bigg\}^n_{j = 1} < \theta_0 \; ,\; \sum^n_{i = 1} \textup{ arccot } \lambda_i < \Theta_0 \right\} 
,
\end{equation*} 
without loss of generality. 
Then we introduce an approximation equation for $0 < t \leq 1$ and nonnegative smooth function $f$, 
\begin{equation}
\label{equation-1-5}
\begin{aligned}
&\mathfrak{Re} \left(\chi + \tilde \chi + t\omega + \sqrt{-1} \partial\bar\partial \varphi_t + \sqrt{-1} \omega\right)^n \\
&\qquad = 
\cot (\theta_0) \mathfrak{Im} \left(\chi + \tilde \chi + t\omega + \sqrt{-1} \partial\bar\partial \varphi_t + \sqrt{-1} \omega\right)^n + c_t f \omega^n ,
\end{aligned}
\end{equation}
where 
$\sup_M \varphi_t = 0$, 
$\int_M f\omega^n = \int_M \omega^n$ and
\begin{equation}
\mathfrak{Re} \int_M\left(\chi + \tilde \chi + t\omega + \sqrt{-1} \omega\right)^n 
= 
\cot (\theta_0) \mathfrak{Im} \int_M\left(\chi  + \tilde \chi + t\omega + \sqrt{-1} \omega\right)^n + c_t \int_M \omega^n.
\end{equation}
It is easy to see that $c_t$ is an increasing non-negative coefficient with respect to parameter $t$.
In fact, for $ t > 0$
\begin{equation}
\begin{aligned}
\dfrac{\partial c_t }{\partial t} 
&=
\dfrac{n}{\int_M \omega^n}   \int_M \bigg(\mathfrak{Re} (\chi + \tilde \chi +t\omega +  \sqrt{-1} \partial\bar\partial \underline{u}_t + \sqrt{-1} \omega)^{n - 1} \\
&\qquad \qquad \qquad \quad - \cot (\theta_0) \mathfrak{Im} (\chi + \tilde \chi + t\omega + \sqrt{-1} \partial\bar\partial \underline{u}_t +\sqrt{-1} \omega)^{n - 1} \bigg)\wedge \omega  
> 0
,
\end{aligned}
\end{equation}
where $\tilde \chi + t\omega + \sqrt{-1} \partial\bar\partial \underline{u} _t > 0$. 
%
%
%
%
%
%
%
%
%
%
%
%
By the work of Chen~\cite{Chen2021}, there is a smooth solution to Equation~\eqref{equation-1-5}, which is also to be discussed later.

The main result of this paper is as follows. 
\begin{theorem}
\label{main-theorem}

Suppose that smooth function $f \in L^q (M) $ for some $q > 1$.
The supercritical phase case of approximation equation~\eqref{equation-1-5} admits a unique smooth solution $\varphi_t$ for all $0 < t \leq 1$.
There is a sequence $\{t_i\} \subset (0,1]$ decreasing to $0$ such that 
$\varphi_{t_i}$ pointwisely converges to a $(\chi + \tilde \chi - \cot (\theta_0) \omega)$-PSH function $\varphi$,
if either of the following conditions holds true: 
\begin{enumerate}
\item $n \geq 4$ and $\theta_0 \in (0, \pi)$; 

\item $n = 3$ and $\theta_0 \in \left(0, \dfrac{\pi}{2}\right]$; 

\item $n = 3$, $\theta_0 \in \left(\dfrac{\pi}{2}, \pi\right)$ and $\bm{\lambda} (\chi + \sqrt{-1} \partial\bar\partial v) \in \bar \Gamma^2$.

\end{enumerate}

\end{theorem}
%
%
%
%
%

When the envelope $U$ in \eqref{definition-2-2} is bounded, 
we can see that $\{\varphi_t\}_{0 < t \leq 1}$ is uniformly bounded
and hence $\left\{\varphi_{t_i} + \dfrac{C}{2^{i + 1}}\right\}$ is decreasing for some constant $C$. 
Therefore, the limit function $\varphi$ is indeed a weak solution in pluripotential sense.
\begin{corollary}
Suppose that 
$U$ is bounded. 
The supercritical phase case of Equation~\eqref{equation-1-2-1} admits a bounded pluripotential solution $\varphi$ which is $(\chi + \tilde \chi - \cot(\theta_0) \omega)$-PSH, 
\begin{enumerate}
\item $n \geq 4$ and $\theta_0 \in (0, \pi)$; 

\item $n = 3$ and $\theta_0 \in \left(0, \dfrac{\pi}{2}\right]$; 

\item $n = 3$, $\theta_0 \in \left(\dfrac{\pi}{2}, \pi\right)$ and $\bm{\lambda} (\chi + \sqrt{-1} \partial\bar\partial v) \in \bar \Gamma^2$.

\end{enumerate}

\end{corollary}


The key assumption is the existence of $\tilde \chi$. 
It is very likely that we can derive some numerical characterizations of nef class $[\tilde \chi]$, in views of~\cite{DemaillyPaun2004}\cite{Chen2021}. 
Meanwhile, there is no way to numerically characterize sempositive $\tilde \chi$ on general K\"ahler manifolds. 
However, there is a easy sufficient condition, that is,
class $\left[\chi - \cot \left(\dfrac{\theta_0}{n - 1}\right)\omega\right]$ is big and has a $C^2$ semipositive representative form 
in Equation~\eqref{equation-1-1}.
%
%
%


\medskip
\section{Preliminary}
\label{preliminary}
 
In this section, we shall state some notations, lemmas and theorems.

\medskip

\subsection{Elementary lemmas}


To deal with the boundary case, we shall adapt the argument of Guo-Phong-Tong~\cite{GPT2021} to an extended $\mathcal{C}$-subsolution condition~\cite{Sun202305}. The argument of Guo-Phong-Tong also works on nef classes~\cite{GPTW2021}. 
By discovering appropriate $\mathcal{C}$-subsolution conditions,  this technique can be applied to various complex equations~\cite{GPT2021}\cite{GPTW2021}\cite{GP2022}\cite{SuiSun2021}\cite{Sun202210}\cite{Sun202211}\cite{Sun202301}\cite{Sun202305}.
A key step in the argument is from Wang-Wang-Zhou~\cite{WangWangZhou2021}, who utilized a De Giorgi iteration. 
In this paper, we shall adopt the following lemma on De Giorgi iteration from \cite{ChenWu}\cite{GilbargTrudinger}. 
\begin{lemma}
\label{lemma-2-2}
Suppose that $\phi (s): [s_0,+\infty) \to [0,+\infty)$ is an  increasing function such that
\begin{equation*}
	s'  \phi (s' + s) \leq C \phi^{1 + \delta} (s), \qquad \forall s' >0, s\geq s_0
\end{equation*}
for some positive constant   $\delta$. Then $\phi (s_0 + d) = 0$ whenever $d \geq C \phi^{\delta} (s_0) 2^{\dfrac{1 + \delta}{\delta}}$ .
\end{lemma}
We shall use the iteration to obtain $L^\infty$ estimates and stability estimates. Then a convergent decreasing function sequence can be constructed, 
and the corresponding limit function can be viewed as a weak solution in pluripotential sense 
if the sequence is uniformly bounded.

%
%
%
%

The envelope of class $[\tilde \chi + t\omega]$  is defined by
\begin{equation}
	U_t := \sup \left\{u | \tilde \chi + t\omega + \sqrt{-1} \partial\bar\partial u \geq 0 \text{ in  current sense} , u \leq 0 \right\} ,
\end{equation}
which might not be  smooth. 
Berman~\cite{Berman2019} constructed a smooth approximation for $U_t$.
\begin{lemma}
\label{lemma-Berman-approximation}
Let $u_\beta$ be the unique smooth admissible solution to the complex Monge-Amp\`ere equation
\begin{equation*}
	\left(\tilde\chi + t\omega + \sqrt{-1} \partial\bar\partial u \right)^n = e^{\beta u } \omega^n.
\end{equation*}
Then we have 
\begin{equation*}
\lim_{\beta \to + \infty} \Vert u_\beta - U_t \Vert_{L^\infty} = 0 .
\end{equation*}

\end{lemma}

Let $0 < t_1 < t_2$.
Since
$\tilde \chi + t_2 \omega + \sqrt{-1} \partial \bar\partial u 
> 
\tilde \chi + t_1 \omega + \sqrt{-1} \partial \bar\partial u$,
we can conclude that $U_{t_1} \leq U_{t_2}$.
Boucksom~\cite{Boucksom2004} (see also \cite{DemaillyPaun2004})
proved that there is a function $\rho$ and a constant $\kappa > 0$ such that $\rho$ is smooth in  $Amp (\tilde\chi)$ with analytic singularities, $\sup_M \rho = 0$
and
$\tilde\chi + \sqrt{-1}\partial\bar\partial\rho \geq \kappa \omega $ in the current sense. 
Then it is reasonable to define
\begin{equation}
\label{definition-2-2}
	U := \lim_{t\to 0+ } U_t \geq \rho .
\end{equation}

\medskip
\subsection{Properties of deformed Hermitian-Yang-Mills equation and its approximation equation}
We can express the terms in the equations by
\begin{equation}
\mathfrak{Re} \left(\prod^n_{i = 1} (\lambda_i + \sqrt{-1})\right) = \cos \left(\sum^n_{i = 1}  \textup{arccot }\lambda_i \right) \prod^n_{i = 1} \sqrt{1 + \lambda^2_i} ,
\end{equation}
and
\begin{equation}
\mathfrak{Im} \left(\prod^n_{i = 1} (\lambda_i + \sqrt{-1})\right) = \sin \left(\sum^n_{i = 1}  \textup{arccot }\lambda_i \right) \prod^n_{i = 1} \sqrt{1 + \lambda^2_i} .
\end{equation}

The function $S_k$ is the $k$-th elementary polynomial, that is
\begin{equation*}
S_k (\bm{\lambda}) = S_k (\lambda_1,\cdots,\lambda_n) = \sum_{i_1 < i_2 < \cdots < i_k} \lambda_{i_1} \lambda_{i_2} \cdots \lambda_{i_k} .
\end{equation*}
The $k$-positive cone $\Gamma^k \subset \mathbb{R}^n$ is defined as
\begin{equation*}
	\Gamma^k := \{\bm{\lambda} \in \mathbb{R}^n \,|\, S_1 (\bm{\lambda}) > 0, \cdots , S_k (\bm{\lambda}) > 0\}
	.
\end{equation*}
The left term in \eqref{equation-1-2} has the following properties discovered by Wang-Yuan~\cite{WangYuan2014}.
\begin{lemma}
\label{lemma-2-3-1}
Suppose that $\lambda_1 \geq \lambda_2 \geq \cdots \geq \lambda_n$ satisfying 
\begin{equation*}
\sum^n_{i = 1} \textup{ arccot } \lambda_i \leq \pi .
\end{equation*}
Then $(\lambda_1, \cdots , \lambda_n) \in \bar \Gamma^k$, 
$
\lambda_1 \geq \lambda_2 \geq \cdots \geq \lambda_{n - 1} > 0 
$, 
and
$
\lambda_1 +  (n - 1)\lambda_n \geq 0
$
.

\end{lemma}

By the continuity and monotonicity of $\sum^n_{i = 1} \textup{ arccot }  \lambda_i$, we see that $(\lambda_1 ,\cdots, \lambda_n ) \in \Gamma^k$,
\begin{equation*}
\lambda_1 \geq \lambda_2 \geq \cdots \geq \lambda_{n - 1} > 0 
,
\qquad
\text{ and }
\qquad
\lambda_1 +  (n - 1)\lambda_n > 0
,
\end{equation*}
when 
\begin{equation}
\sum^n_{i = 1} \textup{ arccot } \lambda_i < \pi .
\end{equation}

%
%
%
%

In this paper, we shall also need to well utilize the properties of approximation equation~\cite{Chen2021}. 
For convenience, we include the statement here. 
\begin{lemma}[Lemma 5.6 in \cite{Chen2021}]
\label{lemma-2-3}

Let $b$ be a parameter, and we define on $\bar \Gamma_{\theta_0, \Theta_0}$ that
\begin{equation*}
\mathfrak{g} (\bm{\lambda}) 
:= 
\dfrac{\mathfrak{Re} \left(\prod^n_{i = 1} (\lambda_i + \sqrt{-1})\right)}{\mathfrak{Im} \left(\prod^n_{i = 1} (\lambda_i + \sqrt{-1})\right)}
- 
\dfrac{b}{\mathfrak{Im} \left(\prod^n_{i = 1} (\lambda_i + \sqrt{-1})\right)}
.
\end{equation*}
%
%
%
There exist positive constants $\epsilon_1 (n,\theta_0,\Theta_0)$, $\epsilon_2 (n,\Theta_0)$ and $C (n,\Theta_0)$
such that function $\mathfrak{g}$ satisfies the following properties:
if  $b \geq 0$ when $n = 1, 2 ,3$ or $b \geq - \epsilon_1$ when $n \geq 4$, then
\begin{enumerate}
\item $\mathfrak{Im} \left(\prod^n_{i = 1} (\lambda_i + \sqrt{-1})\right) \geq C(n,\Theta_0) $ ;\\

\item $\left| \dfrac{\partial}{\partial \lambda_i} \dfrac{1}{\mathfrak{Im} \left(\prod^n_{i = 1} (\lambda_i + \sqrt{-1})\right)}\right| \leq \dfrac{1}{\sqrt{C (n,\Theta_0)}}  \sqrt{\dfrac{\prod^n_{i = 1} (1 + \lambda^2_i)}{\left(\mathfrak{Im} \left(\prod^n_{i = 1} (\lambda_i + \sqrt{-1})\right)\right)^3}} \dfrac{1}{\sqrt{1 + \lambda^2_i}}$ ;\\

\item $\dfrac{\partial\mathfrak{g}}{\partial\lambda_i} > 0$;\\

\item when $n \geq 4$, 
$\left[\dfrac{\partial^2 \mathfrak{g}}{\partial\lambda_i\partial\lambda_j}\right] \leq - \epsilon_2 \dfrac{\prod^n_{i = 1} (1 + \lambda^2_i)}{\left(\mathfrak{Im} \left(\prod^n_{i = 1} (\lambda_i + \sqrt{-1})\right)\right)^3} \left[ \dfrac{\delta_{ij}}{1 + \lambda^2_i}  \right]$; \\
when $n = 1, 2, 3$, 
$\left[\dfrac{\partial^2 \mathfrak{g}}{\partial\lambda_i\partial\lambda_j}\right] \leq 0$; \\

\item if $\bm{\lambda} \in \bar \Gamma_{\theta_0,\Theta_0}$ and $\mathfrak{g} (\bm{\lambda}) = \cot (\theta_0)$, then $\bm{\lambda} \in \Gamma_{\theta_0,\Theta_0}$;\\

\item for any $\bm{\lambda} \in \Gamma_{\theta_0,\Theta_0}$, the set
$$
\left\{ \bm{\lambda}' \in \Gamma_{\theta_0,\Theta_0} \big| \mathfrak{g} (\bm{\lambda}') = 0, \; \lambda'_i \geq \lambda_i , \; \forall i = 1,2,\cdots, n\right\}
$$
is bounded, where the bound depends on $n$, $\theta_0$, $\Theta_0$, $\bm{\lambda}$, $|b|$;\\

\item $\bar \Gamma_{\theta_0,\Theta_0}$ is convex.

%
%
%
%
%

\end{enumerate}

\end{lemma}
The proof is pretty lengthy. For details, we refer the readers to \cite{Chen2021}.

\medskip
\section{$L^\infty$ estimate}
\label{L-estimate}

In this section, we shall prove the $L^\infty$ estimate for approximation equation.
In this paper, some notations may vary in different places, e.g. $C$, $\hat\chi$, $A_{s,k,\beta}$, etc. 
But these notations are clearly stated in each argument, without any confusion.

From the boundary case of $\mathcal{C}$-subsolution condition~\eqref{boundary-case-2-5}, we know that 
\begin{equation}
\mathfrak{Re} (\chi + \sqrt{-1} \partial\bar\partial v +  \sqrt{-1} \omega)^m
\geq 
\cot (\theta_0) \mathfrak{Im} (\chi + \sqrt{-1} \partial\bar\partial v+ \sqrt{-1} \omega)^m
,
\quad
\forall 1 \leq  m \leq n - 1
.
\end{equation}
%
%
%
%
%
%
Suppose that $\hat\chi := \tilde \chi + t\omega + \sqrt{-1} \partial\bar\partial (\varphi_t - v) \geq 0$ at some point $z \in M$.
By expansion, 
 \begin{equation}
 \label{equality-2-3}
 \begin{aligned}
 LHS
 &=
  \sum^n_{i = 0} C^i_n \hat\chi^i \wedge  \mathfrak{Re}  \left(\chi + \sqrt{-1} \partial\bar\partial v + \sqrt{-1} \omega\right)^{n - i} \\
 &=
  \hat\chi^n + \sum^{n - 1}_{i = 0} C^i_n \hat\chi^i \wedge  \mathfrak{Re}  \left(\chi + \sqrt{-1} \partial\bar\partial v + \sqrt{-1} \omega\right)^{n - i} 
 \end{aligned}
 \end{equation}
and
\begin{equation}
\label{equality-2-4}
\begin{aligned}
RHS
&= 
\cot (\theta_0) \sum^n_{i = 0} C^i_n \hat\chi^i \wedge  \mathfrak{Im}   \left(\chi + \sqrt{-1} \partial\bar\partial v + \sqrt{-1} \omega\right)^{n - i}  + c_t f \omega^n \\
&= 
\cot (\theta_0) \sum^{n - 1}_{i = 0} C^i_n \hat\chi^i \wedge  \mathfrak{Im}   \left(\chi + \sqrt{-1} \partial\bar\partial v + \sqrt{-1} \omega\right)^{n - i}  + c_t f \omega^n 
.
\end{aligned}
\end{equation}
Combining \eqref{equality-2-3} and \eqref{equality-2-4}, we obtain that
\begin{equation}
\label{inequality-3-4}
\begin{aligned}
\hat\chi^n
&= 
\sum^{n - 1}_{i = 0} C^i_n \hat\chi^i \wedge \bigg(\cot (\theta_0) \mathfrak{Im}   \left(\chi + \sqrt{-1} \partial\bar\partial v + \sqrt{-1} \omega\right)^{n - i}  \\
&\qquad\qquad\qquad\qquad - \mathfrak{Re}   \left(\chi + \sqrt{-1} \partial\bar\partial v + \sqrt{-1} \omega\right)^{n - i}   \bigg) + c_t f\omega^n \\
&\leq 
 \cot (\theta_0) \mathfrak{Im}   \left(\chi + \sqrt{-1} \partial\bar\partial v + \sqrt{-1} \omega\right)^{n } - \mathfrak{Re}   \left(\chi + \sqrt{-1} \partial\bar\partial v + \sqrt{-1} \omega\right)^{n }   + c_t f\omega^n \\
 &\leq 
 \left(C (\chi, \omega, v) + c_t f \right) \omega^n 
 .
\end{aligned}
\end{equation}
%
For simplicity, we may use $\varphi_t$ to replace $\varphi_t - v$ in the following argument.  
Indeed, we may assume that $v \equiv 0$ in the following sections, without loss of generality.

The following complex Monge-Amp\`ere equation has a smooth admissible solution~\cite{Yau1978}:
\begin{equation}
	\left(\tilde \chi + t\omega + \sqrt{-1} \partial\bar\partial \psi_{s,k}\right)^n
	=
	\dfrac{\tau_k (-\varphi_t + u_\beta - s)}{A_{s,k,\beta}} \left(C(\chi,\omega,v) + c_t \mathcal{F}_k\right) \omega^n ,
	\quad
	\sup_M \psi_{s,k} = 0 ,
\end{equation}
where $M_s := \{ - \varphi_t + U_t - s > 0\}$, $V_t = \int_M (\tilde \chi + t\omega)^n$
and 
$u_\beta$  is from Lemma~\ref{lemma-Berman-approximation}.
Function $\tau_k : \mathbb{R} \to \mathbb{R}^+$ is a uniformly decreasing sequence of smooth functions such that
\begin{equation*}
	\max\{t,0\} + \dfrac{1}{k} \leq \tau_k (t) \leq \max\{t,0\} + \dfrac{2}{k}, 
\end{equation*}
and $\mathcal{F}_k$ is a uniformly decreasing sequence of smooth functions on $M$ such that
\begin{equation*}
	0 < \mathcal{F}_k - f - \epsilon < \dfrac{1}{k} ,
\end{equation*}
where $1 > \epsilon > 0$.  
Moreover, 
\begin{equation*}
\begin{aligned}
	A_{s,k,\beta} 
	&:= \dfrac{1}{V_t} \int_{M}   \tau_k (- \varphi_t + u_\beta - s)  \left(C(\chi,\omega,v) + c_t \mathcal{F}_k \right) \omega^n &&\qquad\\
	\to 
	A_{s,k}
	&:= \dfrac{1}{V_t} \int_{M}    \tau_k (- \varphi_t + U_t - s)  \left(C(\chi,\omega,v) + c_t \mathcal{F}_k \right) \omega^n
	&&\qquad (\beta \to +\infty)
	\\
	\to
	A_s 
	&:= \dfrac{1}{V_t} \int_{M_s} (- \varphi_t + U_t - s) g_{t,\epsilon} \omega^n  
	&&\qquad (k\to \infty)
	\\
	\leq	E_t
	&:= \dfrac{1}{V_t} \int_{M_0} (-\varphi_t + U_t) g_{t,\epsilon} \omega^n 
	,
\end{aligned}
\end{equation*}
where 
\begin{equation*}
g_{t,\epsilon} := C(\chi,\omega,v) + c_t (f + \epsilon) > 0 .
\end{equation*}
In particular, $\psi_{s,k} \leq  U_t \leq 0$.

We consider the function
\begin{equation*}
	 -   A^{\frac{1}{n + 1}}_{s,k,\beta}\left( \dfrac{n + 1}{n}\left(- \psi_{s,k} + u_\beta + 1\right) +  A_{s,k,\beta}\right)^{\frac{n}{n + 1}} - (\varphi_t - u_\beta + s) .
\end{equation*}
As in \cite{GPT2021}\cite{GPTW2021}\cite{SuiSun2021}, it can be proven that 
\begin{equation*}
	- \varphi_t + u_\beta - s
	\leq 
	  A^{\frac{1}{n + 1}}_{s,k,\beta}\left(\dfrac{n + 1}{n}\left(- \psi_{s,k} + u_\beta + 1\right) +  A_{s,k,\beta}\right)^{\frac{n}{n + 1}} + \Vert u_\beta - U_t\Vert_{L^\infty} ,
\end{equation*}
when $\beta$ is sufficiently large.
Letting $\beta \to +\infty$,
\begin{equation}
\label{inequality-3-5}
	- \varphi_t + U_t - s
	\leq 
	  A^{\frac{1}{n + 1}}_{s,k}\left(\dfrac{n + 1}{n}\left(- \psi_{s,k} + U_t + 1\right) +  A_{s,k}\right)^{\frac{n}{n + 1}}  
	  .
\end{equation}
As shown in \cite{Hormander}\cite{Tian1987}, there exist $\alpha_0 > $ and $C > 0$ such that
\begin{equation}
\label{inequality-3-7}
\begin{aligned}
	&\quad \int_{M_s} \exp \left( \alpha_0 \dfrac{(-\varphi_t + U_t - s)^{\frac{n + 1}{n}}}{A^{\frac{1}{n}}_{s,k}} \right) \omega^n 
	&\leq 
	C \exp \left(\alpha_0 A_{s,k}\right)
	.
\end{aligned}
\end{equation}
Letting $k \to \infty$ in \eqref{inequality-3-7}, 
\begin{equation}
\label{inequality-3-8}
	\int_{M_s} \exp \left(  \dfrac{\alpha_0 (-\varphi_t + U_t - s)^{\frac{n + 1}{n}}}{A^{\frac{1}{n}}_{s}} \right) \omega^n 
	\leq 
	C \exp \left(\alpha_0 A_{s}\right) 
	\leq 
	C \exp \left(\alpha_0 E_t\right)
	.
\end{equation}
By generalized Young's inequality and \eqref{inequality-3-8}, we derive that
\begin{equation}
\label{inequality-3-9}
\begin{aligned}
	&\quad \dfrac{\alpha^p_0}{2^p A^{\frac{p}{n}}_s} \int_{M_s} (-\varphi_t + U_t - s)^{\frac{(n + 1) p}{n}} g_{t,\epsilon} \omega^n 
	\\
	&\leq
	\int_{M_s} g_{t,\epsilon} \ln^p \left(1 + g_{t,\epsilon} \right) \omega^n   + C \int_{M_s} \exp \left( \dfrac{\alpha_0 (-\varphi_t + U_t - s)^{\frac{n + 1}{n}}}{A^{\frac{1}{n}}_s}\right) \omega^n \\
	&\leq 
	\int_{M_s} g_{t,\epsilon} \ln^p g_{t,\epsilon} \omega^n   + 	C \exp \left(\alpha_0 E_t\right)
	.
\end{aligned}
\end{equation}
Applying H\"older inequality with respect to measure $g_{t,\epsilon} \omega^n$ and \eqref{inequality-3-9} to quantity $A_s$,
\begin{equation}
\label{inequality-3-10}
\begin{aligned}
	A_s 
	&\leq 
	\dfrac{1}{V_t} \left(  \int_{M_s} (- \varphi_t + U_t - s)^{\frac{(n + 1) p}{n}} g_{t,\epsilon} \omega^n\right)^{\frac{n}{(n + 1) p}}  \left( \int_{M_s} g_{t,\epsilon} \omega^n\right)^{\frac{(n + 1) p - n}{(n + 1) p}} \\
	&\leq 
	\dfrac{1}{V_t} \left(\dfrac{2 C_t A^{\frac{1}{n}}_s}{\alpha_0}\right)^{\frac{n}{n + 1}} \left(  \int_{M_s} g_{t,\epsilon} \omega^n\right)^{\frac{(n + 1) p - n}{(n + 1) p}}
		,
\end{aligned}
\end{equation}
where
\begin{equation}
\begin{aligned}
	C_t &:= \Bigg( 	\int_{M} g_{t,\epsilon} \ln^p \left(1 + g_{t,\epsilon} \right) \omega^n   + 	C \exp \left(\alpha_0 E_t\right)\Bigg)^{\frac{1}{ p}}
\end{aligned}
\end{equation}
Then by rewriting \eqref{inequality-3-10},
\begin{equation}
	A_s 
	\leq \dfrac{2 C_t}{\alpha_0 V^{ \frac{ n + 1}{n}}_t} \left(  \int_{M_s} g_{t,\epsilon} \omega^n\right)^{1 + \frac{1}{n} - \frac{1}{p}} .
\end{equation}
For any $s', s > 0$,
\begin{equation}
\label{inequality-3-13}
\begin{aligned}
	s' \int_{M_{s + s'}} g_{t,\epsilon} \omega^n 
	&\leq  \dfrac{2 C_t}{\alpha_0 V^{ \frac{ 1}{n}}_t} \left(  \int_{M_s} g_{t,\epsilon} \omega^n\right)^{1 + \frac{1}{n} - \frac{1}{p}} 
		.
\end{aligned}
\end{equation}
By Lemma~\ref{lemma-2-2} and \eqref{inequality-3-13},  we obtain that
\begin{equation}
\label{inequality-3-13-1}
		{-\varphi_t + U_t}  
		\leq    2^{\frac{np + 2 p - 2 n}{p - n}} \dfrac{ C_t}{\alpha_0 V^{ \frac{ 1}{n}}_t} \left(  \int_{M} g_{t,\epsilon} \omega^n\right)^{ \frac{1}{n} - \frac{1}{p}} ,
\end{equation}
when $p > n$.

To find the $t$-independent $L^\infty$ estimate, it suffices to find a uniform upper bound for $E_t$, which is a component of coefficient $C_t$ in \eqref{inequality-3-13-1}.  
We shall adapt the argument of Guo-Phong~\cite{GP202207}. 
From \eqref{inequality-3-5}, on $M_s$
\begin{equation}
\label{inequality-3-14}
\begin{aligned}
	(- \varphi_t + U_t - s)^{\frac{p (n + 1)}{n}}  g_{t,\epsilon}
	&\leq 
	   \left(\dfrac{n + 1}{n}\left(- \psi_{s,k} + U_t + 1\right) +  A_{s,k}\right)^p  A^{\frac{p}{n}}_{s,k}  g_{t,\epsilon}
	   \\
	 &\leq 
	 	C \left((-\psi_{s,k} + U_t + 1)^p  A^{\frac{p}{n}}_{s,k}  + A^{\frac{(n + 1) p}{n}}_{s,k}\right) g_{t,\epsilon}
	  .
\end{aligned}
\end{equation}
Applying generalized Young's inequality to \eqref{inequality-3-14},
\begin{equation}
\label{inequality-3-15}
\begin{aligned}
	&\quad 
	(- \varphi_t + U_t - s)^{\frac{p (n + 1)}{n}}  g_{t,\epsilon}
	\\
	&\leq  
	C \Bigg(g_{t,\epsilon} \ln^p \left(1 + g_{t,\epsilon} \right)  + \exp \left(   \alpha_0 \dfrac{n + 1}{n} \left( - \psi_{s,k} + U_t + 1\right) \right)  \Bigg) A^{\frac{p}{n}}_{s,k}
	+ C A^{\frac{(n + 1) p}{n}}_{s,k}  g_{t,\epsilon}
	.
\end{aligned}
\end{equation}
Integrating \eqref{inequality-3-15} over $M_s$, we obtain that
\begin{equation}
\label{inequality-3-16}
\begin{aligned}
	&\quad 
	\int_{M_s}(- \varphi_t + U_t - s)^{\frac{p (n + 1)}{n}}  g_{t,\epsilon} \omega^n
	\\
	&\leq  
	C A^{\frac{p}{n}}_{s,k} \int_{M_s} g_{t,\epsilon}  \ln^p \left(1 + g_{t,\epsilon} \right) \omega^n    + C A^{\frac{p}{n}}_{s,k} \int_{M_s} \exp \left(   \alpha_0 \dfrac{n + 1}{n} \left( - \psi_{s,k} + U_t + 1\right) \right)    \omega^n \\
	&\qquad 
	+ C A^{\frac{(n + 1) p}{n}}_{s,k} \int_{M_s} g_{t,\epsilon} \omega^n \\
	&\leq 
	C A^{\frac{p}{n}}_{s,k} \int_{M_0} g_{t,\epsilon}  \ln^p \left(1 + g_{t,\epsilon} \right) \omega^n  
	+ C A^{\frac{p}{n}}_{s,k} 
	+ C A^{\frac{(n + 1) p}{n}}_{s,k} \int_{M_s} g_{t,\epsilon} \omega^n 
	.
\end{aligned}
\end{equation}
Letting $k\to \infty$ in \eqref{inequality-3-16},
\begin{equation}
\label{inequality-3-17}
\begin{aligned}
	&\quad 
	\int_{M_s}(- \varphi_t + U_t - s)^{\frac{p (n + 1)}{n}}  g_{t,\epsilon} \omega^n
	\\ 
	&\leq 
	C A^{\frac{p}{n}}_{s} \int_{M_0} g_{t,\epsilon} \ln^p \left(1 + g_{t,\epsilon} \right) \omega^n   + C A^{\frac{p}{n}}_{s}  + C A^{\frac{(n + 1) p}{n}}_s \int_{M_s} g_{t,\epsilon} \omega^n 
	.
\end{aligned}
\end{equation}
Applying \eqref{inequality-3-17} and H\"older inequality to quantity $A_s$, 
\begin{equation*}
\begin{aligned}
	A_s
	&\leq
	\dfrac{1}{V_t} \left(\int_{M_s} (- \varphi_t + U_t - s)^{\frac{p (n + 1)}{n}}  g_{t,\epsilon} \omega^n\right)^{\frac{n}{p (n + 1)}}  \left(\int_{M_s}  g_{t,\epsilon} \omega^n\right)^{1 - \frac{n}{p (n + 1)}} \\
	&\leq \dfrac{C_1 A^{\frac{1}{n + 1}}_{s}}{V_t} \Bigg( \int_{M_0} g_{t,\epsilon} \ln^p \left(1 + g_{t,\epsilon} \right) \omega^n  + 1 + A^p_{s} \int_{M_s} g_{t,\epsilon} \omega^n \Bigg)^{\frac{n}{p (n + 1)}}  \left(\int_{M_s}  g_{t,\epsilon} \omega^n\right)^{1 - \frac{n}{p (n + 1)}} 
	,
\end{aligned}
\end{equation*}
and hence
\begin{equation}
\label{inequality-3-18}
\begin{aligned}
	&\quad 
	\left(V^{\frac{p (n + 1)}{n}}_t  -  C_1 \left(\int_{M_s}  g_{t,\epsilon} \omega^n\right)^{\frac{p (n + 1)}{n} } \right) A^p_s
	\\
	&\leq 
	C_1 \Bigg( \int_{M_0} g_{t,\epsilon} \ln^p \left(1 + g_{t,\epsilon} \right) \omega^n + 1   \Bigg)   \left(\int_{M_s}  g_{t,\epsilon} \omega^n\right)^{\frac{p (n + 1)}{n} - 1}
	.
\end{aligned}
\end{equation}
According to the proof of the decay estimate in Guo-Phong~\cite{GP2022} and the fact revealed by Wang-Yuan~\cite{WangYuan2014} that 
\begin{equation}
\label{inequality-3-1-19}
\left(\chi + \tilde \chi + t\omega + \sqrt{-1}\partial\bar\partial \varphi_t\right) \wedge \omega^{n - 1} 
\geq
0 ,
\end{equation}
we obtain that for $s > 1$, 
\begin{equation}
\label{inequality-3-19}
\begin{aligned}
	\int_{M_s}  g_{t,\epsilon} \omega^n 
	&\leq 
	\int_{M_s} \left(\dfrac{\ln \left(-\varphi_t + U_t \right)}{\ln s} \right)^p g_{t,\epsilon} \omega^n \\
	&\leq 
	\dfrac{2^p}{\ln^p s} \int_{M_s} \bigg( g_{t,\epsilon}  \ln^p \left(1 + g_{t,\epsilon} \right)  + C_p (- \varphi_t + U_t)\bigg)\omega^n \\	
	&\leq 
	\dfrac{C_2}{\ln^p s} \left( \int_{M_0}  g_{t,\epsilon}  \ln^p \left(1 + g_{t,\epsilon} \right) \omega^n + 1\right)
	.
\end{aligned}
\end{equation}
We write down the details  of \eqref{inequality-3-19} here to understand the influence of $\epsilon$. 
Choosing 
\begin{equation*}
\begin{aligned}
	s \geq s_1
	&:=
	1 + \exp \Bigg(\dfrac{ (2 C_1)^{\frac{n}{p (n + 1)}} C_2}{V_t}   \left( \int_{M_0}  g_{t,\epsilon}  \ln^p \left(1 + g_{t,\epsilon} \right) \omega^n + 1\right)   \Bigg)^{\frac{1}{p}}
	,
\end{aligned}
\end{equation*}
we obtain from \eqref{inequality-3-18} and \eqref{inequality-3-19} that
\begin{equation}
\begin{aligned}
	   A_s 
	&\leq 
	\dfrac{(2 C_1)^{\frac{1}{p}}}{V^{\frac{n + 1}{n}}_t} \Bigg( \int_{M_0} g_{t,\epsilon}  \ln^p \left(1 + g_{t,\epsilon} \right) \omega^n + 1   \Bigg)^{\frac{1}{p}}  \left(\int_{M_s}  g_{t,\epsilon} \omega^n\right)^{1 + \frac{1}{n} - \frac{1}{p}} \\
	&\leq 
	\dfrac{(2 C_1)^{\frac{1}{p}} C^{1 + \frac{1}{n} - \frac{1}{p}}_2}{V^{\frac{n + 1}{n}}_t (\ln s)^{p + \frac{p}{n} - 1 } } \left( \int_{M_0} g_{t,\epsilon}  \ln^p \left(1 + g_{t,\epsilon} \right) \omega^n + 1   \right)^{1 + \frac{1}{n} } \\
	&<
	\dfrac{(2 C_1)^{\frac{n}{p^2 (n + 1)}}  }{V^{\frac{1}{p}}_t  } \left( \int_{M_0} g_{t,\epsilon}  \ln^p \left(1 + g_{t,\epsilon} \right) \omega^n + 1   \right)^{\frac{1}{p} } 
	.
\end{aligned}
\end{equation}
Therefore, we have a uniform upper bound for $E_t$, 
\begin{equation}
\label{inequality-3-21}
\begin{aligned}
E_t
&\leq
A_{s_1}   + \dfrac{s_1}{V_t} \int_{M_0}  g_{t,\epsilon} \omega^n \\
&\leq 
\dfrac{(2 C_1)^{\frac{n}{p^2 (n + 1)}}  }{V^{\frac{1}{p}}_0  } \left( \int_{M} g_{1,\epsilon}  \ln^p \left(1 + g_{1,\epsilon} \right) \omega^n + 1   \right)^{\frac{1}{p} } \\
&\qquad +
\dfrac{	1 + \exp \Bigg(\dfrac{ (2 C_1)^{\frac{n}{p (n + 1)}} C_2}{V_0}   \left( \int_{M}  g_{1,\epsilon}  \ln^p \left(1 + g_{1,\epsilon} \right) \omega^n + 1\right)   \Bigg)^{\frac{1}{p}}}{V_0} \int_{M}  g_{1,\epsilon} \omega^n 
.
\end{aligned}
\end{equation}
Substituting \eqref{inequality-3-21} into \eqref{inequality-3-13-1} and letting $\epsilon \to 0+$, we obtain a uniform upper bound for $-\varphi_t + U_t$ when $0 < t \leq 1$.

\begin{theorem}
Suppose that in Equation~\eqref{equation-1-5},  $f \ln^p (1 + f)$ is integrable for some $p > n$. 
Let $\varphi_t$ be the solution to  Equation~\eqref{equation-1-5}. 
Then there exists a constant $K_0 > 0$ such that $- \varphi_t + U_t < K_0$ for all $t \in (0,1]$.
\end{theorem}

\begin{remark}
Given that $U$ is bounded, we can conclude that $\varphi_t$ is uniformly bounded for $0 < t \leq 1$.
\end{remark}


\medskip

\medskip
\section{Solvability of approximation equation}
\label{solvability}

In the following argument, 
we require that $[X_{i\bar j}]$ is diagonal and $\omega_{i\bar j} = \delta_{ij}$ at the point we do calculation. 
We may further assume that $X_{1\bar 1} \geq \cdots \geq X_{n\bar n}$, for simplicity.

Let $I$ be the set of $t$ such that there exists a smooth solution $\varphi_t$ to Equation~\eqref{equation-1-5} satisfying
\begin{equation} 
\bm{\lambda} (\chi + \tilde \chi + t \omega + \sqrt{-1} \partial\bar\partial \varphi_t) 
\in
\Gamma_{\theta_0,\Theta_0}
.
\end{equation} 
We plan to prove that $(0,+\infty) \subset I$, 
which is equivalent to $[2\epsilon,+\infty) \subset I$ for any small $\epsilon > 0$.

Since $\bm{\lambda} (\chi + \tilde \chi + \omega) \in \Gamma_{\theta_0,\Theta_0}$ and $b_t f \geq 0$, there exists a smooth solution $\varphi_1$ to Equation~\eqref{equation-1-5} 
with $\bm{\lambda} (\chi + \tilde \chi + \omega + \sqrt{-1} \partial\bar\partial \varphi_1) \in \Gamma_{\theta_0,\Theta_0}$,
and then we know that $[1,+\infty) \subset I$.
Therefore, we only need to show that $[3\epsilon , 1] \subset I$ for any $\epsilon \in \left(0,\dfrac{1}{3}\right)$.

There is a smooth function $v_\epsilon$ such that
\begin{equation}
	\tilde \chi + \epsilon \omega + \sqrt{-1} \partial\bar\partial v_\epsilon 
	> 
	0
	,
\end{equation}
as $\tilde \chi$ is nef.
Consequently,
\begin{equation}
	\tilde \chi + t\omega + \sqrt{-1} \partial\bar\partial v_\epsilon 
	> 
	(t - \epsilon) \omega 
	\geq 
	2 \epsilon \omega
	,
	\qquad 
	\forall t \in [3\epsilon , +\infty)
	.
\end{equation}
Indeed, $v_\epsilon$ is an $\mathcal{C}$-subsolution. 
Then, 
\begin{equation}
\label{inequality-7-12}
\begin{aligned}
\sum_i \dfrac{\partial\mathfrak{f}}{\partial \lambda_i}  (v_{\epsilon, i\bar i} - \varphi_{t,i\bar i})
&=
\sum_i \dfrac{\partial\mathfrak{f}}{\partial \lambda_i} (\chi_{i\bar i} + \tilde \chi_{i\bar i} + t  + v_{\epsilon,i\bar i} - X_{i\bar i})
\\
&\geq
\sum_i \dfrac{\partial\mathfrak{f}}{\partial \lambda_i} (\chi_{i\bar i} + 2 \epsilon   - X_{i\bar i})
\\
&=
\sum_i \dfrac{\partial\mathfrak{f}}{\partial \lambda_i} \tilde X_{i\bar i} - \sum_i \dfrac{\partial\mathfrak{f}}{\partial \lambda_i} X_{i\bar i} - N \dfrac{\partial\mathfrak{f}}{\partial \lambda_1} + \epsilon \sum_i \dfrac{\partial\mathfrak{f}}{\partial \lambda_i}
\\
&\geq
\mathfrak{F} (\tilde X) - \cot (\theta_0) - N \dfrac{\partial\mathfrak{f}}{\partial \lambda_1} + \epsilon \sum_i \dfrac{\partial\mathfrak{f}}{\partial \lambda_i}
,
\end{aligned} 
\end{equation}
where
$\tilde X := \chi + \epsilon \omega  + N \sqrt{-1} d z^1 \wedge d \bar z^1 $. 
Since
\begin{equation}
\label{inequality-4-7-7}
\sum_{i \neq j} \textup{ arccot } \left(\lambda_i (\tilde X)\right) <  \sum_{i\neq j} \textup{ arccot } \left(\lambda_i (\chi)\right)  \leq \theta_0  ,
\end{equation}
we can derive that
\begin{equation}
\label{inequality-4-7-8}
\sum_i \textup{ arccot } \left(\lambda_i (\tilde X)\right) < \theta_0 ,
\end{equation}
when $N$ is sufficiently large.
The last line of \eqref{inequality-7-12} is due to concavity of $\mathfrak{f}$, \eqref{inequality-4-7-7} and \eqref{inequality-4-7-8}. 
By Inequality~\eqref{boundary-case-2-5},  
we can see that there is a constant $\xi \in (0, 1)$ such that
\begin{equation}
\begin{aligned}
\sum_{k\neq j} \textup{arccot } \left(\lambda_k \left(\chi \right) + \epsilon\right) 
< 
\textup{arccot }\left(\cot (\theta_0) + 2 \xi\right),\qquad \forall j = 1,2,\cdots, n.
\end{aligned}
\end{equation}
Therefore, 
\begin{equation}
\label{inequality-7-15}
\begin{aligned}
&\quad
	\mathfrak{Re} \left(\tilde X + \sqrt{-1} \omega\right)^n - \left(\cot (\theta_0) + \xi\right) \mathfrak{Im} \left(\tilde X + \sqrt{-1} \omega\right)^n  
	- c_t f \omega^n
	\\
	&\geq
	\mathfrak{Re} \left(\chi + \epsilon \omega   + \sqrt{-1} \omega\right)^n - (\cot (\theta_0) + \xi) \mathfrak{Im} \left(\chi + \epsilon \omega  + \sqrt{-1} \omega\right)^n 
	- c_t f \omega^n \\
	&\qquad
	+
	n   N \sqrt{-1} d z^1 \wedge d\bar z^1 \\
	&\qquad\qquad
	\wedge \left(\mathfrak{Re} \left(\chi + \epsilon \omega   + \sqrt{-1} \omega\right)^{n - 1} - (\cot (\theta_0) + \xi) \mathfrak{Im} \left(\chi + \epsilon \omega   + \sqrt{-1} \omega\right)^{n - 1}\right)
	\\
	&>
	\mathfrak{Re} \left(\chi + \epsilon \omega   + \sqrt{-1} \omega\right)^n - (\cot (\theta_0) + \xi) \mathfrak{Im} \left(\chi + \epsilon \omega  + \sqrt{-1} \omega\right)^n
	- c_t f \omega^n
	\\
	&\qquad
	+
	\xi n N \sqrt{-1} d z^1 \wedge d\bar z^1  
	\wedge  \mathfrak{Im} \left(\chi + \epsilon \omega  + \sqrt{-1} \omega\right)^{n - 1} 
	\\
	&> 0
	,
\end{aligned}
\end{equation}
when $N  $ is sufficiently large.
Inequality~\eqref{inequality-7-15} can be rewritten as
\begin{equation}
\label{inequality-7-15-1}
	\mathfrak{F} (\tilde X) > \cot (\theta_0) + \xi .
\end{equation}
Substituting \eqref{inequality-7-15-1} into \eqref{inequality-7-12}, 
\begin{equation}
\label{inequality-7-16}
\sum_i \dfrac{\partial\mathfrak{f}}{\partial \lambda_i}  (v_{\epsilon,i\bar i} - \varphi_{t,i\bar i})
>
\xi - N \dfrac{\partial\mathfrak{f}}{\partial \lambda_1} + \epsilon \sum_i \dfrac{\partial\mathfrak{f}}{\partial \lambda_i} 
.
\end{equation}
Supposing that
\begin{equation}
	\dfrac{\partial \mathfrak{f}}{\partial \lambda_1} 
	\leq 
	\dfrac{1}{2 N}
	\min \left\{\epsilon,\xi\right\} \left(1 + \sum_i \dfrac{\partial \mathfrak{f}}{\partial \lambda_i}\right)
	,
\end{equation}
we derive from \eqref{inequality-7-16} that
\begin{equation}
\label{inequality-7-15-2}
	\sum_i \dfrac{\partial\mathfrak{f}}{\partial \lambda_i}  (v_{\epsilon,i\bar i} - \varphi_{t,i\bar i})
	> 
	\dfrac{1}{2} \min \left\{\epsilon,\xi\right\} \left(1 + \sum_i \dfrac{\partial \mathfrak{f}}{\partial \lambda_i}\right)
	.
\end{equation}
%
%
Inequality~\eqref{inequality-7-15-2} can help us to derive partial $C^2$ estimates and then $C^\infty$ estimates, as in \cite{Szekelyhidi2018}\cite{CollinsJacobYau2020}\cite{Chen2021}.

By the implicit function theorem, $I \cap [2\epsilon, 1]$ is non-empty and open. 
If there is a sequence $\{t_i\} \subset I \cap [2\epsilon,1]$ such that 
$\lim_{i\to\infty} t_i = T' \in [2\epsilon, 1]$,
then we can take a limit $\varphi_{T'} \in C^\infty (M)$ by Arzela-Ascoli Theorem such that
\begin{equation}
\begin{aligned}
&\mathfrak{Re} \left(\chi + \tilde \chi + T' \omega + \sqrt{-1} \partial\bar\partial \varphi_{T'} + \sqrt{-1} \omega\right)^n \\
&\qquad = 
\cot (\theta_0) \mathfrak{Im} \left(\chi + \tilde \chi + T' \omega + \sqrt{-1} \partial\bar\partial \varphi_{T'} + \sqrt{-1} \omega\right)^n + c_{T'} f \omega^n .
\end{aligned}
\end{equation}
In particular,
$\bm{\lambda} (\chi + \tilde \chi + T' \omega + \sqrt{-1} \partial\bar\partial \varphi_{T'})  \in \Gamma_{\theta_0,\Theta_0}$.
We obtain $ [2\epsilon , 1 ]  \subset I$ and hence $(0,+ \infty) \subset I$.

In previous works, it is required that
\begin{equation}
\label{inequality-4-17}
\sum^n_{i = 1} \textup{ arccot } \left(\lambda_i (\chi + \sqrt{-1} \partial\bar\partial v)\right) < \Theta_0,
\end{equation}
in addition to $\mathcal{C}$-subsolution~\eqref{inequality-2-2}.
According to the above argument, we can remove condition~\eqref{inequality-4-17}. 
\begin{proposition}
\label{proposition-4-1}
Suppose that there exists a real-valued $C^2$ function $v$ satisfying $\mathcal{C}$-subsolution condition~\eqref{inequality-2-2}. 
Then there exists a unique smooth solution $\varphi$ solving the supercritical phase case of Equation~\eqref{equation-1-1}.
\end{proposition}
\begin{proof}

We shall simply consider that for $t \in (0,T]$,
\begin{equation}
\label{equation-4-7-17}
\begin{aligned}
&\mathfrak{Re} \left(\chi + t\omega + \sqrt{-1} \partial\bar\partial \varphi_t + \sqrt{-1} \omega\right)^n  
\\
&= 
\cot (\theta_0) \mathfrak{Im} \left(\chi + t\omega + \sqrt{-1} \partial\bar\partial \varphi_t + \sqrt{-1} \omega\right)^n + c_t  \omega^n ,
\qquad
\sup_M \varphi_t = 0
,
\end{aligned}
\end{equation}
where
\begin{equation}
\sum^n_{i = 1} \textup{ arccot } \left(\lambda_i  (\chi + T \omega )\right) < \Theta_0 .
\end{equation}
Let $I$ be the set of $t$ such that there exists a smooth solution $\varphi_t$. 
It is easy to see that $T \in I$ by Chen~\cite{Chen2021}.
Define
\begin{equation}
\mathfrak{f} := \cot \left(\sum_k \textup{arccot }\lambda_k\right) - \dfrac{c_t}{\sin \left(\sum_k \textup{arccot } \lambda_k\right) \prod^n_{k = 1} \sqrt{1 + \lambda^2_k}} ,
\end{equation}
and Equation~\eqref{equation-4-7-17} can be expressed as
\begin{equation}
	\mathfrak{F} (\chi + t\omega + \sqrt{-1}\partial\bar\partial \varphi_t) 
	:=
	\mathfrak{f} (\bm{\lambda} (\chi + t\omega + \sqrt{-1} \partial\bar\partial \varphi_t)) 
	= \cot(\theta_0) .
\end{equation}

From \eqref{inequality-2-2}, we know that there is a constant $\epsilon > 0$ and $\xi > 0$ such that
\begin{equation}
\sum_{i\neq j} \textup{ arccot } (\lambda_i (\chi - \epsilon \omega)) < \textup{arccot } ( \cot(\theta_0) + 2 \xi), \qquad \forall j = 1,2,\cdots,n.
\end{equation}
When $N$ is sufficiently large,
\begin{equation}
	\sum_i \textup{arccot } (\lambda_i (\tilde X)) < \theta_0 ,
\end{equation}
where $\tilde X := \chi - \epsilon \omega + N \sqrt{-1} d z^1 \wedge d\bar z^1$. 
Similar to \eqref{inequality-7-15}, 
\begin{equation}
\begin{aligned}
&\quad
	\mathfrak{Re} \left(\tilde X + \sqrt{-1} \omega\right)^n - \left(\cot (\theta_0) + \xi\right) \mathfrak{Im} \left(\tilde X + \sqrt{-1} \omega\right)^n  
	- c_t  \omega^n
	\\
	&\geq
	\mathfrak{Re} \left(\chi - \epsilon \omega   + \sqrt{-1} \omega\right)^n - (\cot (\theta_0) + \xi) \mathfrak{Im} \left(\chi - \epsilon \omega  + \sqrt{-1} \omega\right)^n 
	- c_t  \omega^n \\
	&\qquad
	+
	n   N \sqrt{-1} d z^1 \wedge d\bar z^1 \\
	&\qquad\qquad
	\wedge \left(\mathfrak{Re} \left(\chi - \epsilon \omega   + \sqrt{-1} \omega\right)^{n - 1} - (\cot (\theta_0) + \xi) \mathfrak{Im} \left(\chi - \epsilon \omega   + \sqrt{-1} \omega\right)^{n - 1}\right)
	\\
	&>
	\mathfrak{Re} \left(\chi - \epsilon \omega   + \sqrt{-1} \omega\right)^n - (\cot (\theta_0) + \xi) \mathfrak{Im} \left(\chi - \epsilon \omega  + \sqrt{-1} \omega\right)^n
	- c_t  \omega^n
	\\
	&\qquad
	+
	\xi n N \sqrt{-1} d z^1 \wedge d\bar z^1  
	\wedge  \mathfrak{Im} \left(\chi - \epsilon \omega  + \sqrt{-1} \omega\right)^{n - 1} 
	\\
	&> 0
	,
\end{aligned}
\end{equation}
when $N$ is sufficiently large. 
Therefore
\begin{equation}
\begin{aligned}
	- \sum_i \dfrac{\partial \mathfrak{f}}{\partial \lambda_i} \varphi_{t,i\bar i}
	&\geq 
	\sum_i \dfrac{\partial \mathfrak{f}}{\partial \lambda_i} \left(\chi_{i\bar i} - \epsilon\right) + N \dfrac{\partial \mathfrak{f}}{\partial \lambda_1}  - \sum_i \dfrac{\partial \mathfrak{f}}{\partial \lambda_i} X_{i\bar i} + \epsilon \sum_i \dfrac{\partial\mathfrak{f}}{\partial\lambda_i} - N \dfrac{\partial\mathfrak{f}}{\partial\lambda_1}
	\\
	&\geq 
	\mathfrak{F} (\tilde X) - \cot (\theta_0) + \epsilon \sum_i \dfrac{\partial\mathfrak{f}}{\partial\lambda_i} - N \dfrac{\partial\mathfrak{f}}{\partial\lambda_1}
	\\
	&> 
	\xi + \epsilon \sum_i \dfrac{\partial\mathfrak{f}}{\partial\lambda_i} - N \dfrac{\partial\mathfrak{f}}{\partial\lambda_1}
	.
\end{aligned}
\end{equation}
For $t > 0$, 
we have
\begin{equation*}
	\sum_i \dfrac{\partial\mathfrak{f}}{\partial \lambda_i}  (v_{\epsilon,i\bar i} - \varphi_{t,i\bar i})
	> 
	\dfrac{1}{2} \min \left\{\epsilon,\xi\right\} \left(1 + \sum_i \dfrac{\partial \mathfrak{f}}{\partial \lambda_i}\right)
	,
\end{equation*}
when
\begin{equation*}
	\dfrac{\partial \mathfrak{f}}{\partial \lambda_1} 
	\leq 
	\dfrac{1}{2 N}
	\min \left\{\epsilon,\xi\right\} \left(1 + \sum_i \dfrac{\partial \mathfrak{f}}{\partial \lambda_i}\right)
	.
\end{equation*}

Similar to the previous continuity method argument, we can obtain that $[0,T] \subset I$.

\end{proof}

\medskip
\section{Stability estimate for $n\geq 4$}
\label{stability}

In this section, we shall study the stability estimate for approximation equation~\eqref{equation-1-5} for the case of $n \geq 4$.
For $0 < t \leq 1$, we assume that
\begin{equation}
\label{equality-4-1}
\begin{aligned}
&\mathfrak{Re} \left(\chi + \tilde \chi + t \omega + \sqrt{-1} \partial\bar\partial \varphi_1 + \sqrt{-1} \omega\right)^n \\
&\qquad = 
\cot (\theta_0) \mathfrak{Im} \left(\chi + \tilde \chi + t \omega + \sqrt{-1} \partial\bar\partial \varphi_1 + \sqrt{-1} \omega\right)^n + c_t f_1 \omega^n ,
\qquad \sup_M \varphi_1 = 0,
\end{aligned}
\end{equation}
and
\begin{equation}
\label{equality-4-2}
\begin{aligned}
&\mathfrak{Re} \left(\chi + \tilde \chi + t  \omega + \sqrt{-1} \partial\bar\partial \varphi_2 + \sqrt{-1} \omega\right)^n \\
&\qquad = 
\cot (\theta_0) \mathfrak{Im} \left(\chi + \tilde \chi + t  \omega + \sqrt{-1} \partial\bar\partial \varphi_2 + \sqrt{-1} \omega\right)^n + c_t f_2 \omega^n ,
\qquad \sup_M \varphi_2 = 0,
\end{aligned}
\end{equation}
for nonnegative smooth function $f_1 \in L^q$ and $f_2 \in L^1$ with $\int_M f_1\omega^n = \int_M f_2 \omega^n = \int_M \omega^n$.
%
%
%
%
%
%
In case $\Vert (\varphi_2 - \varphi_1) \Vert_{L^{q^*}} = 0$, it is easy to see that $\varphi_2 \leq \varphi_1$, which implies that the the stability estimate holds true for any positive coefficient. In the remaining of stability estimate, we only need to prove the stability estimate in the case of $\Vert (\varphi_2 - \varphi_1)^+ \Vert_{L^{q^*}} > 0$.

\medskip
\subsection{An intermediate function}
Choosing an appropriate positive constant $\sigma_1 < \epsilon_1$, 
we can define constant $ s_1 \in \left(\dfrac{1}{2},1\right)$ by
\begin{equation}
\mathfrak{Re} \int_M\left(\chi + s_1 \tilde \chi + \sqrt{-1} \omega\right)^n 
= 
\cot (\theta_0) \mathfrak{Im} \int_M\left(\chi  + s_1 \tilde \chi + \sqrt{-1} \omega\right)^n 
- \sigma_1 \int_M \omega^n
,
\end{equation}
and hence constant $ T_1 > 0$ by
\begin{equation}
\mathfrak{Re} \int_M\left(\chi + s_1 \tilde \chi + T_1 \omega + \sqrt{-1} \omega\right)^n 
= 
\cot (\theta_0) \mathfrak{Im} \int_M\left(\chi  + s_1 \tilde \chi + T_1 \omega + \sqrt{-1} \omega\right)^n 
.
\end{equation}
We try to solve the supercritical phase case of
\begin{equation}
\label{equation-4-3}
\begin{aligned}
&\quad
\mathfrak{Re}  \left(\chi + s_1 ( \tilde \chi + t \omega) + \sqrt{-1} \partial\bar\partial v_t + \sqrt{-1} \omega\right)^n \\
&= 
\cot (\theta_0) \mathfrak{Im}  \left(\chi  + s_1 (\tilde \chi + t \omega) + \sqrt{-1} \partial\bar\partial v_t + \sqrt{-1} \omega\right)^n 
- \sigma_1  \omega^n + b_t \omega^n
,
\end{aligned}
\end{equation}
where $\sup_M v_t = 0$ and
\begin{equation}
\begin{aligned}
&\quad
\mathfrak{Re}  \int_M \left(\chi + s_1 (\tilde \chi + t \omega)   + \sqrt{-1} \omega\right)^n \\
&= 
\cot (\theta_0) \mathfrak{Im}  \int_M \left(\chi  + s_1 (\tilde \chi + t \omega)  + \sqrt{-1} \omega\right)^n 
- \sigma_1  \int_M \omega^n + b_t \int_M\omega^n
.
\end{aligned}
\end{equation}
In particular, we know that $b_t > 0$ and $\dfrac{\partial b_t}{\partial t} > 0 $ for $t > 0$ as is $c_t$. 
As in Section~\ref{L-estimate}, we can derive that there is a constant $K_1 > 0$ such that
\begin{equation}
	- v_t + s_1 U_t \leq K_1 
\end{equation}
for any $0 < t \leq 1$.

Similar to Section~\ref{solvability}, we can prove that Equation~\eqref{equation-4-3} admits a smooth solution for all $t > 0$. 
Let $I$ be the set of $t$ such that there exists a smooth solution $v_t$ to Equation~\eqref{equation-4-3} satisfying
\begin{equation}
\label{4-5}
\bm{\lambda} (\chi + s_1 ( \tilde \chi + t \omega ) + \sqrt{-1} \partial\bar\partial v_t) \in \Gamma_{\theta_0,\Theta_0} .
\end{equation} 
It suffices to prove that $(0,+\infty) \subset I$.
If $t \geq \dfrac{T_1}{s_1}$, 
there is a unique smooth solution to Equation~\eqref{equation-4-3}, 
by the works of Collins-Jacob-Yau~\cite{CollinsJacobYau2020},Chen~\cite{Chen2021} and Proposition~\ref{proposition-4-1}. 
So $ \left[\dfrac{T_1}{s_1} , +\infty\right) \subset I$, and we only need to prove that 
$\left[\epsilon , \dfrac{T_1}{s_1}\right]  \subset I$ for any positive constant $  \epsilon < \dfrac{T_1}{s_1} $.
Given that $t \in \left[\epsilon , \dfrac{T_1}{s_1}\right]$, 
it is easy to see that $b_t - \sigma_1 \subset [b_\epsilon- \sigma_1, 0]$,
and consequently there are uniform $C^\infty$ a priori estimates for smooth solutions to Equation~\eqref{equation-4-3} satisfying the relation~\eqref{4-5}.
By the implicit function theorem, $I \cap \left[\epsilon, \dfrac{T_1}{s_1}\right]$ is non-empty and open.
If there is a sequence $\{t_i\} \subset I \cap \left[\epsilon, \dfrac{T_1}{s_1}\right]$ such that $\lim_{i\to \infty} t_i = T' \in \left[\epsilon, \dfrac{T_1}{s_1} \right]$, 
then we can take a  limit $v_{T'}$ such that
\begin{equation}
\begin{aligned}
&\quad
\mathfrak{Re}  \left(\chi + s_1 (\tilde \chi + T' \omega ) + \sqrt{-1} \partial\bar\partial v_{T'} + \sqrt{-1} \omega\right)^n \\
&= 
\cot (\theta_0) \mathfrak{Im}  \left(\chi  + s_1 ( \tilde \chi + T' \omega ) + \sqrt{-1} \partial\bar\partial v_{T'} + \sqrt{-1} \omega\right)^n 
- \sigma_1  \omega^n + b_{T'} \omega^n
,
\end{aligned}
\end{equation}
since there are uniform $C^\infty$ a priori estimates for $\{v_{t_i}\}$.
In particular,
\begin{equation}
\bm{\lambda} (\chi + s_1 (\tilde \chi + T'\omega ) + \sqrt{-1} \partial\bar\partial v_{T'}) \in \bar \Gamma_{\theta_0, \Theta_0}.
\end{equation}
By Chen~\cite{Chen2021}, $\bm{\lambda} (\chi + s_1 (\tilde \chi + T' \omega ) + \sqrt{-1} \partial\bar\partial v_{T'}) \in \Gamma_{\theta_0, \Theta_0}$. 
So, $T' \in I \cap \left[\epsilon, \dfrac{T_1}{s_1}\right]$ and $I \cap \left[\epsilon, \dfrac{T_1}{s_1}\right]$ is closed. 
Therefore, $I \cap \left[\epsilon, \dfrac{T_1}{s_1}\right] = \left[\epsilon, \dfrac{T_1}{s_1}\right]$.

\medskip

\subsection{The estimate of difference $v_t - \varphi_1$}
The argument is very similar to that of $L^\infty$ estimate in Section~\ref{L-estimate}. 
We shall reuse some notations without mentioning the definitions, e.g. $\tau_k$, $U_t$, etc.

Comparing \eqref{equality-4-1} and \eqref{equation-4-3}, 
\begin{equation}
\begin{aligned}
	c_t f_1 \omega^n
	&=
\mathfrak{Re} \left(\chi + \tilde \chi + t \omega + \sqrt{-1} \partial\bar\partial \varphi_1 + \sqrt{-1} \omega\right)^n \\
&\qquad 
-
\cot (\theta_0) \mathfrak{Im} \left(\chi + \tilde \chi + t \omega + \sqrt{-1} \partial\bar\partial \varphi_1 + \sqrt{-1} \omega\right)^n 
\\
	&= 
\mathfrak{Re} \left(\chi + s_1 (\tilde \chi + t \omega) + \sqrt{-1} \partial\bar\partial v_t + \hat\chi + \sqrt{-1} \omega\right)^n \\
&\qquad 
-
\cot (\theta_0) \mathfrak{Im} \left(\chi + s_1(\tilde \chi + t \omega) + \sqrt{-1} \partial\bar\partial v_t + \hat\chi + \sqrt{-1} \omega\right)^n 
\\
	&=
	\sum^{n - 1}_{i = 1} C^i_n \hat\chi^i \wedge \Bigg(
	\mathfrak{Re} (\chi + s_1 (\tilde \chi + t\omega) + \sqrt{-1} \partial\bar\partial v_t + \sqrt{-1} \omega)^{n - i}
	\\
	&\qquad \qquad \qquad \qquad
	-
	\cot (\theta_0) \mathfrak{Im} (\chi + s_1 (\tilde \chi + t\omega) + \sqrt{-1} \partial\bar\partial v_t + \sqrt{-1} \omega)^{n - i}
	\Bigg)
	\\
	&\qquad
	+ 
	\hat\chi^n 
	+ 
	\mathfrak{Re} (\chi + s_1 (\tilde \chi + t\omega) + \sqrt{-1} \partial\bar\partial v_t + \sqrt{-1} \omega)^n \\
	&\qquad \qquad \qquad
	-
	\cot (\theta_0) \mathfrak{Im} (\chi + s_1 (\tilde \chi + t\omega) + \sqrt{-1} \partial\bar\partial v_t + \sqrt{-1} \omega)^n \\
	&\geq \hat\chi^n - \sigma_1 \omega^n + b_t \omega^n 
	,
\end{aligned}
\end{equation}
wherever
\begin{equation}
\hat\chi := (1 - s_1) (\tilde \chi + t\omega) + \sqrt{- 1} \partial\bar\partial (\varphi_1 - v_t) .
\end{equation}

We solve the auxiliary complex Monge-Amp\`ere equation
\begin{equation}
	\left( (1 - s_1) (\tilde \chi + t\omega) + \sqrt{-1} \partial\bar\partial \psi_{s,k} \right)^n
	=
	\dfrac{\tau_k \left( v_t - \varphi_1 + (1 - s_1) u_\beta - s\right)}{A_{s,k,\beta}} (c_t f_1 + \sigma_1)\omega^n
	,
\end{equation}
where
\begin{equation}
A_{s,k,\beta} 
:= 
\dfrac{1}{ (1 - s_1)^n V_t}  \int_M \tau_k \left( v_t - \varphi_1 + (1 - s_1) u_\beta - s\right)  (c_t f_1 + \sigma_1) \omega^n
.
\end{equation}
We shall study the following function,
\begin{equation}
	- \left(\dfrac{n + 1}{n} A^{\frac{1}{n}}_{s,k,\beta} \left( - \psi_{s,k} + (1 - s_1) (u_\beta + 1) \right) + A^{\frac{n + 1}{n}}_{s,k,\beta}\right)^{\frac{n}{n + 1}} 
	+ v_t - \varphi_1 + (1 - s_1) u_\beta - s .
\end{equation}
We can prove as in the previous section that in $M$,
\begin{equation}
\label{inequality-4-16}
\begin{aligned}
	v_t - \varphi_1 + (1 - s_1) U_t - s 
	\leq \left(\frac{n + 1}{n} A^{\frac{1}{n}}_{s,k} \left( - \psi_{s,k} + (1 - s_1) (U_t + 1) \right) + A^{\frac{n + 1}{n}}_{s,k}\right)^{\frac{n}{n + 1}} 
	,
\end{aligned}
\end{equation}
where
\begin{equation}
A_{s,k} := \dfrac{1}{(1 - s_1)^n V_t} \int_M \tau_k \left( v_t - \varphi_1 + (1 - s_1) U_t - s\right)  (c_t f_1 + \sigma_1) \omega^n. 
\end{equation}
Integrating \eqref{inequality-4-16} over $M_s$, 
\begin{equation}
\begin{aligned}
&\quad 
\int_{M_s} \exp \left( \dfrac{\alpha_0}{1 - s_1} \dfrac{\left(v_t - \varphi_1 + (1 - s_1) U_t - s\right)^{\frac{n  +1}{n}}}{A^{\frac{1}{n}}_{s,k}} \right)\omega^n 
\\
&\leq 
\int_{M_s} \exp \left( \dfrac{\alpha_0}{ 1 - s_1} \left( \dfrac{n + 1}{n} (- \psi_{s,k} + 1 - s_1) + A_{s,k}\right)\right) \omega^n
\\
&\leq C \exp \left(\dfrac{\alpha_0}{1 - s_1}A_{s,k}\right)
,
\end{aligned}
\end{equation}
and hence
\begin{equation}
\int_{M_s} \exp \left( \dfrac{\alpha_0}{1 - s_1} \dfrac{\left(v_t - \varphi_1 + (1 - s_1) U_t - s\right)^{\frac{n  +1}{n}}}{A^{\frac{1}{n}}_s} \right)\omega^n 
\leq
C \exp \left(\dfrac{\alpha_0}{1 - s_1} A_s \right)
,
\end{equation}
where
$M_s := \{v_t - \varphi_1 + (1 - s_1) U_t \geq s\}$
and
\begin{equation}
A_s := 
\dfrac{1}{(1 - s_1)^n V_t} \int_{M_s}  \left( v_t - \varphi_1 + (1 - s_1) U_t - s\right)  (c_t f_1 + \sigma_1) \omega^n 
.
\end{equation}
It is easy to see that
\begin{equation}
\label{bound-4-21}
\begin{aligned}
A_s 
&\leq
\dfrac{s_1}{(1 - s_1)^n V_t} \int_{M}  - U_t   (c_t f_1 + \sigma_1) \omega^n 
 + \dfrac{ \Vert (- \varphi + U_t)^+ \Vert_{L^\infty} }{(1 - s_1)^n V_t}   \int_M (c_t f_1 + \sigma_1) \omega^n
 \\
&\leq 
\dfrac{s_1}{(1 - s_1)^n V_t} \left(\int_{M}  (c_t f_1 + \sigma_1) \ln^p (1 + c_t f_1 + \sigma_1)\omega^n + C \int_M \exp \left(2 (- U_t)^{\frac{1}{p}}\right) \omega^n \right) \\
&\qquad
 + \dfrac{ \Vert (- \varphi + U_t)^+ \Vert_{L^\infty} }{(1 - s_1)^n V_t}   \int_M (c_t f_1 + \sigma_1) \omega^n
\\
&\leq 
\dfrac{s_1}{(1 - s_1)^n V_t} \left(\int_{M}  (c_t f_1 + \sigma_1) \ln^p (1 + c_t f_1 + \sigma_1)\omega^n 
+ C_3 \right) \\
&\qquad
 + \dfrac{ \Vert (- \varphi + U_t)^+ \Vert_{L^\infty} }{(1 - s_1)^n V_t}   \int_M (c_t f_1 + \sigma_1) \omega^n
. 
\end{aligned}
\end{equation}
We define
\begin{equation}
\begin{aligned}
	E_t 
	&:= \dfrac{s_1}{(1 - s_1)^n V_t} \left(\int_{M}  (c_t f_1 + \sigma_1) \ln^p (1 + c_t f_1 + \sigma_1)\omega^n 
	+ C_3 \right)
	\\
	&\qquad
	 + \dfrac{ \Vert (- \varphi + U_t)^+ \Vert_{L^\infty} }{(1 - s_1)^n V_t}   \int_M (c_t f_1 + \sigma_1) \omega^n
	 .
\end{aligned}
\end{equation}
Then we can derive
\begin{equation}
\label{inequality-4-22}
\begin{aligned}
A_s
&\leq 
\dfrac{1}{(1 - s_1)^n V_t}  \left(\int_{M_s} \left(v_t - \varphi_1 + (1 - s_1) U_t - s\right)^{\frac{(n + 1) p}{n}} (c_t f_1 + \sigma_1) \omega^n\right)^{\frac{n}{(n + 1) p}} 
\\
&\qquad
\cdot
\left(\int_{M_s} (c_t f_1 + \sigma_1) \omega^n \right)^{\frac{(n + 1) p - n}{(n + 1) p}}
\\
&\leq
\dfrac{ (1 - s_1)^{\frac{n}{n + 1}} 2^{\frac{n}{n + 1}} A^{\frac{1}{n + 1}}_s}{ \alpha^{\frac{n}{n + 1}}_0 (1 - s_1)^n V_t} \\
&\qquad
\cdot 
\left( \int_{M_s} (c_t f_1 + \sigma_1) \ln^p (1 + c_t f_1 + \sigma_1) \omega^n + C \exp \left(\dfrac{\alpha_0}{1 - s_1} A_s\right)\right)^{\frac{n}{(n + 1) p}} 
\\
&\qquad
\cdot
\left(\int_{M_s} (c_t f_1 + \sigma_1) \omega^n \right)^{\frac{(n + 1) p - n}{(n + 1) p}}
\\
&\leq
\dfrac{C^{\frac{n}{n + 1}}_t}{(1 - s_1)^{\frac{n^2}{n + 1} } V_t}  A^{\frac{1}{n + 1}}_s
\left(\int_{M_s} (c_t f_1 + \sigma_1) \omega^n \right)^{\frac{(n + 1) p - n}{(n + 1) p}}
,
\end{aligned}
\end{equation}
where
\begin{equation}
C_t 
:=
\dfrac{2}{\alpha_0  }
\Bigg( \int_{M} (c_t f_1 + \sigma_1) \ln^p (1 + c_t f_1 + \sigma_1) \omega^n 
+ C \exp \left(  \dfrac{\alpha_0  }{1 - s_1} E_t\right)\Bigg)^{\frac{1}{ p}} 
.
\end{equation}
Rewriting \eqref{inequality-4-22},
\begin{equation}
	A_s
	\leq 
	\dfrac{ C_t}{(1 - s_1)^n V^{\frac{n + 1}{n}}_t} 
	\left(\int_{M_s} (c_t f_1 + \sigma_1) \omega^n \right)^{1 + \frac{1}{n} - \frac{1}{p}}
	.
\end{equation}
As in Section~\ref{L-estimate}, we obtain that
\begin{equation}
v_t - \varphi_1 + (1 - s_1) U_t 
\leq 
2^{\frac{np + p - n}{p - n}}\dfrac{C_t }{V^{\frac{1}{n}}_t}
\left( \int_{M} (c_t f_1 + \sigma_1) \omega^n  \right)^{\frac{1}{n} - \frac{1}{p}}
,
\end{equation}
when $p > n$.
It is obvious that $c_t$ is bounded if $f_1 \in L^q$ for  $q > 1$.

\begin{proposition}
There is a constant $K_2 > 0$ such that
\begin{equation}
v_t - \varphi_1 + (1 - s_1) U_t 
\leq 
K_2
,
\end{equation}
for $0 < t \leq 1$.

\end{proposition}

\medskip
\subsection{The stability estimate}

By concavity, we can do the following estimate,  
\begin{equation}
\begin{aligned}
	&
	\dfrac{\mathfrak{Re}	\left(\chi + (1 - r) (\tilde \chi + t\omega + \sqrt{-1} \partial\bar\partial \varphi_2) 
		+ r \left(s_1 (\tilde \chi + t \omega) + \sqrt{-1} \partial\bar\partial v_t\right) + \sqrt{-1} \omega \right)^n }{\mathfrak{Im}	\left(\chi + (1 - r) (\tilde \chi + t\omega + \sqrt{-1} \partial\bar\partial \varphi_2) 
				+ r \left(s_1 (\tilde \chi + t \omega) + \sqrt{-1} \partial\bar\partial v_t\right) + \sqrt{-1} \omega \right)^n} 
				\\
		& + \dfrac{\sigma_1 \omega^n}{\mathfrak{Im}	\left(\chi + (1 - r) (\tilde \chi + t\omega + \sqrt{-1} \partial\bar\partial \varphi_2) 
						+ r \left(s_1 (\tilde \chi + t \omega) + \sqrt{-1} \partial\bar\partial v_t\right) + \sqrt{-1} \omega \right)^n}
		\\
	&\geq (1 - r) \dfrac{\mathfrak{Re}	\left(\chi + \tilde \chi + t\omega + \sqrt{-1} \partial\bar\partial \varphi_2 + \sqrt{-1} \omega \right)^n + \sigma_1 \omega^n}{\mathfrak{Im}	 \left(\chi + \tilde \chi + t\omega + \sqrt{-1} \partial\bar\partial \varphi_2 + \sqrt{-1} \omega \right)^n} \\
	&\qquad 
	+ 
	r \dfrac{\mathfrak{Re}	\left( \chi + s_1 (\tilde \chi + t \omega) + \sqrt{-1} \partial\bar\partial v_t + \sqrt{-1} \omega\right)^n + \sigma_1 \omega^n}{\mathfrak{Im}	\left( \chi + s_1 (\tilde \chi + t  \omega ) + \sqrt{-1} \partial\bar\partial v_t + \sqrt{-1} \omega\right)^n} \\
	&> \cot (\theta_0)
	,
\end{aligned}
\end{equation}
that is, 
\begin{equation}
\label{inequality-4-5}
\begin{aligned}
&\quad
\mathfrak{Re}	\left(\chi + (1 - r) (\tilde \chi + t\omega + \sqrt{-1} \partial\bar\partial \varphi_2) 
		+ r \left(s_1 (\tilde \chi + t \omega) + \sqrt{-1} \partial\bar\partial v_t\right) + \sqrt{-1} \omega \right)^n
		\\
&>
\cot (\theta_0) \mathfrak{Im}	\Bigg(\chi + (1 - r) (\tilde \chi + t\omega + \sqrt{-1} \partial\bar\partial \varphi_2) \\
&\qquad\qquad\qquad\qquad
		+ r \left(s_1 (\tilde \chi + t \omega) + \sqrt{-1} \partial\bar\partial v_t\right) + \sqrt{-1} \omega \Bigg)^n
		- \sigma_1 \omega^n
		.
\end{aligned}
\end{equation}
Then we decompose
\begin{equation}
\label{equality-4-7}
\tilde \chi + t\omega + \sqrt{-1}\partial\bar\partial \varphi_1 
	=
	(1 - r) (\tilde \chi + t\omega + \sqrt{-1} \partial\bar\partial \varphi_2) 
	+ r \left(s_1 (\tilde \chi + t \omega) + \sqrt{-1} \partial\bar\partial v_t\right)
	+ r \hat\chi
	,
\end{equation}
where 
\begin{equation}
\hat\chi 
:= 
(1 - s_1) (\tilde \chi + t \omega ) - \sqrt{-1} \partial\bar\partial v_t + \dfrac{1}{r} \sqrt{-1}\partial\bar\partial \varphi_1 - \dfrac{1 - r}{r} \sqrt{-1} \partial\bar\partial \varphi_2 .
\end{equation}
Calculating the $n$-th exterior power,
\begin{equation}
\label{equality-4-8}
\begin{aligned}
	&\quad 
	\left(\chi + \tilde \chi + t\omega + \sqrt{-1}\partial\bar\partial \varphi_1 + \sqrt{-1} \omega\right)^n
	\\
	&= 
	\sum^{n - 1}_{i = 1}   C^i_n  (r \hat\chi)^i  \wedge \Bigg( \chi + (1 - r) (\tilde \chi + t\omega + \sqrt{-1} \partial\bar\partial \varphi_2) \\
	&\qquad \qquad \qquad \qquad \qquad 
			+ r \left(s_1 (\tilde \chi + t \omega ) + \sqrt{-1} \partial\bar\partial v_t \right) + \sqrt{-1} \omega \Bigg)^{n - i}
	\\
	&\qquad 
	+ r^n \hat\chi^n
	+ \Bigg( \chi + (1 - r) (\tilde \chi + t\omega + \sqrt{-1} \partial\bar\partial \varphi_2) \\
	&\qquad \qquad \qquad \qquad \qquad 
				+ r \left(s_1 ( \tilde \chi +  t \omega  ) + \sqrt{-1} \partial\bar\partial v_t\right) + \sqrt{-1} \omega \Bigg)^n 
	.
\end{aligned}
\end{equation}
At any point in $M$ with $\hat\chi \geq 0$, we obtain from \eqref{equality-4-8}\eqref{inequality-4-5} that 
\begin{equation}
\begin{aligned}
c_t f_1 \omega^n
&=
\mathfrak{Re} 	\left(\chi + \tilde \chi + t\omega + \sqrt{-1}\partial\bar\partial \varphi_1 + \sqrt{-1} \omega\right)^n \\
&\qquad - \cot (\theta_0) \mathfrak{Im} 	\left(\chi + \tilde \chi + t\omega + \sqrt{-1}\partial\bar\partial \varphi_1 + \sqrt{-1} \omega\right)^n \\
	&= 
	\sum^{n - 1}_{i = 1}   C^i_n  (r \hat\chi)^i  \wedge \Bigg(\mathfrak{Re} \bigg( \chi + (1 - r) (\tilde \chi + t\omega + \sqrt{-1} \partial\bar\partial \varphi_2) \\
	&\qquad \qquad \qquad \qquad \qquad  \qquad 
			+ r \left(s_1 (\tilde \chi + t \omega) + \sqrt{-1} \partial\bar\partial v_t \right) + \sqrt{-1} \omega \bigg)^{n - i}
	\\
	&\qquad \qquad 
	- \cot (\theta_0) \mathfrak{Im} \bigg( \chi + (1 - r) (\tilde \chi + t\omega + \sqrt{-1} \partial\bar\partial \varphi_2) \\
		&\qquad \qquad \qquad \qquad \qquad  \qquad
				+ r \left(s_1 (\tilde \chi + t \omega) + \sqrt{-1} \partial\bar\partial v_t \right) + \sqrt{-1} \omega \bigg)^{n - i}
	\Bigg)
	\\
	&\qquad 
	+ r^n \hat\chi^n
	+ \mathfrak{Re}\Bigg( \chi + (1 - r) (\tilde \chi + t\omega + \sqrt{-1} \partial\bar\partial \varphi_2) \\
	&\qquad \qquad \qquad \qquad \qquad \qquad
				+ r \left(s_1 (\tilde \chi + t \omega) + \sqrt{-1} \partial\bar\partial v_t\right) + \sqrt{-1} \omega \Bigg)^n 
	\\
	&\qquad \qquad 
		- \cot (\theta_0) \mathfrak{Im}\Bigg( \chi + (1 - r) (\tilde \chi + t\omega + \sqrt{-1} \partial\bar\partial \varphi_2) \\
		&\qquad \qquad \qquad \qquad \qquad \qquad
					+ r \left(s_1 (\tilde \chi + t \omega) + \sqrt{-1} \partial\bar\partial v_t\right) + \sqrt{-1} \omega \Bigg)^n 
					\\
	&\geq r^n \hat\chi^n - \sigma_1 \omega^n
					.
\end{aligned}
\end{equation}
It holds true that
\begin{equation}
\label{cone-4-16}
\hat\chi^n \leq \dfrac{c_t f_1 + \sigma_1}{r^n} \omega^n ,
\end{equation}
wherever $\hat \chi \geq 0$.

We solve the auxiliary complex Monge-Amp\`ere equation
\begin{equation}
\begin{aligned}
&\quad
	\left( 
	(1 - s_1) (\tilde \chi + t \omega) + \sqrt{-1} \partial\bar\partial \psi_{s,k}
	\right)^n
	\\
	&=
	\dfrac{\tau_k \left( \dfrac{1 - r}{r} \varphi_2 - \dfrac{1}{r} \varphi_1 + v_t + (1 - s_1) u_\beta - s\right)}{A_{s,k,\beta}} \dfrac{c_t f_1 + \sigma_1}{r^n} \omega^n ,
\end{aligned}
\end{equation}
where
\begin{equation}
	A_{s,k,\beta} := \dfrac{1}{ (1 - s_1)^n V_t} \int_M \tau_k \left(\dfrac{1 - r}{r} \varphi_2 - \dfrac{1}{r} \varphi_1 + v_t + (1 - s_1) u_\beta - s\right)  \dfrac{c_t f_1 + \sigma_1}{r^n} \omega^n.
\end{equation}
We study the following function,
\begin{equation}
\begin{aligned}
	\Phi 
	&:= 
	- \left(\dfrac{n + 1}{n} A^{\frac{1}{n}}_{s,k,\beta} \left(-\psi_{s,k} + (1 - s_1) (u_\beta + 1)\right) + A^{\frac{n + 1}{n}}_{s,k,\beta}\right)^{\frac{n}{n + 1}}
	\\
	&\qquad\qquad
	+ \dfrac{1 - r}{r} \varphi_2 - \dfrac{1}{r} \varphi_1 + v_t + (1 - s_1) u_\beta - s ,
\end{aligned}
\end{equation}
where $s \geq K_2$. 
Let 
$M_s := \left\{\dfrac{1 - r}{r} \varphi_2 - \dfrac{1}{r} \varphi_1 + v_t + (1 - s_1) U_t \geq s \right\} $.
We need to be careful about the dependence of the stability estimate, and include the details here.

If the maximal value of $\Phi$ occurs at $\bm{z}_0 \in M\setminus \mathring{M}_s$, 
then
\begin{equation}
\label{inequality-4-20}
\Phi (\bm{z}_0) \leq (1 - s_1) \Vert u_\beta - U_t\Vert_{L^\infty} .
\end{equation}
Otherwise, at $\bm{z}_0 \in \mathring{M}_s$,
\begin{equation}
\label{inequality-4-21}
\begin{aligned}
&\quad
	\sqrt{-1} \partial\bar\partial \left(\dfrac{1}{r} \varphi_1 - \dfrac{1 - r}{r} \varphi_2 - v_t - (1 - s_1) u_\beta \right)
	\\
	&\geq
	\left(\dfrac{n + 1}{n} A^{\frac{1}{n}}_{s,k,\beta} \left(-\psi_{s,k} + (1 - s_1) (u_\beta + 1)\right) + A^{\frac{n + 1}{n}}_{s,k,\beta}\right)^{- \frac{1}{n + 1}} 
	\\
	&\qquad\qquad \cdot  A^{\dfrac{1}{n}}_{s,k,\beta} \sqrt{-1} \partial\bar\partial \left(\psi_{s,k} - (1 - s_1) u_\beta \right) 
	.
\end{aligned}
\end{equation}
We may assume that
\begin{equation}
	- \psi_{s,k} + (1 - s_1) (u_\beta + 1) > - \psi_{s,k} + (1 - s_1) U_t \geq 0 ,
\end{equation}
since $\Vert U_t - u_\beta\Vert_{L^\infty} << 1$ when $\beta$ is sufficiently large.
Then we derive from \eqref{inequality-4-21} that
\begin{equation}
\label{inequality-4-23}
\begin{aligned}
&\quad
	(1 - s_1) (\tilde \chi + t \omega)
	+
	\sqrt{-1} \partial\bar\partial \left(\dfrac{1}{r} \varphi_1 - \dfrac{1 - r}{r} \varphi_2 - v_t   \right)
	\\
	&\geq 
	\left(\dfrac{n + 1}{n} A^{\frac{1}{n}}_{s,k,\beta} \left(-\psi_{s,k} + (1 - s_1) (u_\beta + 1)\right) + A^{\frac{n + 1}{n}}_{s,k,\beta}\right)^{- \frac{1}{n + 1}}  A^{\frac{1}{n}}_{s,k,\beta}
	\\
	&\qquad\qquad \cdot  \left( (1 - s_1) (\tilde \chi + t \omega) + \sqrt{-1} \partial\bar\partial  \psi_{s,k}  \right)
	\\
	&\qquad 
	+
	\left( 1 - 	\left(\dfrac{n + 1}{n} A^{\frac{1}{n}}_{s,k,\beta} \left(-\psi_{s,k} + (1 - s_1)( u_\beta + 1)\right) + A^{\frac{n + 1}{n}}_{s,k,\beta}\right)^{- \frac{1}{n + 1}}  A^{\frac{1}{n}}_{s,k,\beta} \right)
	\\
	&\qquad\qquad \cdot (1 - s_1) (\tilde \chi + t \omega + \sqrt{-1} \partial\bar\partial u_\beta) 
	\\
	&\geq 
	\left(\frac{n + 1}{n} A^{\frac{1}{n}}_{s,k,\beta} \left(-\psi_{s,k} + (1 - s_1) (u_\beta + 1)\right) + A^{\frac{n + 1}{n}}_{s,k,\beta}\right)^{- \frac{1}{n + 1}}  A^{\frac{1}{n}}_{s,k,\beta}
	\\
	&\qquad\qquad \cdot  \left( (1 - s_1) (\tilde \chi + t \omega) + \sqrt{-1} \partial\bar\partial  \psi_{s,k}  \right)
	.
\end{aligned}
\end{equation}
Calculating the $n$-th exterior power of \eqref{inequality-4-23},
\begin{equation}
\label{inequality-4-24}
\begin{aligned}
\hat\chi^n
	&\geq
	\left(\dfrac{n + 1}{n} A^{\frac{1}{n}}_{s,k,\beta} \left(-\psi_{s,k} + (1 - s_1) (u_\beta + 1)\right) + A^{\frac{n + 1}{n}}_{s,k,\beta}\right)^{- \frac{n}{n + 1}}  A_{s,k,\beta}
	\\
	&\qquad\qquad \cdot  \left( (1 - s_1) (\tilde \chi + t \omega) + \sqrt{-1} \partial\bar\partial  \psi_{s,k}  \right)^n
	\\
	&\geq 
	\left(\dfrac{n + 1}{n} A^{\frac{1}{n}}_{s,k,\beta} \left(-\psi_{s,k} + (1 - s_1) (u_\beta + 1)\right) + A^{\frac{n + 1}{n}}_{s,k,\beta}\right)^{- \frac{n}{n + 1}}  
	\\
	&\qquad\qquad \cdot    \left( \dfrac{1 - r}{r} \varphi_2 - \dfrac{1}{r} \varphi_1 + v_t + (1 - s_1) u_\beta - s\right)  \dfrac{c_t f_1 + \sigma_1}{r^n} \omega^n 
	.
\end{aligned}
\end{equation}
Substituting \eqref{cone-4-16} into \eqref{inequality-4-24},  
\begin{equation}
\label{inequality-4-25}
\begin{aligned}
& 
	\left(\dfrac{n + 1}{n} A^{\frac{1}{n}}_{s,k,\beta} \left(-\psi_{s,k} + (1 - s_1) (u_\beta + 1)\right) + A^{\frac{n + 1}{n}}_{s,k,\beta}\right)^{\frac{n}{n + 1}}  \\
&	\geq
   \left( \dfrac{1 - r}{r} \varphi_2 - \dfrac{1}{r} \varphi_1 + v_t + (1 - s_1) u_\beta - s\right)   
	.
\end{aligned}
\end{equation}
We can conclude from \eqref{inequality-4-20} and \eqref{inequality-4-25} that
\begin{equation}
	\Phi \leq (1 - s_1) \Vert u_\beta - U_t\Vert_{L^\infty} \qquad \text{in } M .
\end{equation}
Letting $\beta$ approach $+\infty$, we obtain that
\begin{equation}
\begin{aligned}
&
	\dfrac{1 - r}{r} \varphi_2 - \dfrac{1}{r} \varphi_1 + v_t + (1 - s_1)U_t - s 
	\\
	 &\leq  
	 \left(\dfrac{n + 1}{n} A^{\frac{1}{n}}_{s,k} \left(-\psi_{s,k} + (1 - s_1) (U_t + 1)\right) + A^{\frac{n + 1}{n}}_{s,k }\right)^{\frac{n}{n + 1}}
	.
\end{aligned}
\end{equation}
where 
\begin{equation}
A_{s,k} := \dfrac{1}{ (1 - s_1)^n V_t} \int_M \tau_k \left(\dfrac{1 - r}{r} \varphi_2 - \dfrac{1}{r} \varphi_1 + v_t + (1 - s_1) U_t - s\right)  \dfrac{c_t f_1 + \sigma_1}{r^n} \omega^n
.
\end{equation}
Then
\begin{equation}
\label{inequality-4-29}
\begin{aligned}
&\quad
\int_{M_s} \exp \left( \dfrac{\alpha_0}{ (1 - s_1) A^{\frac{1}{n}}_{s,k}} \left(\dfrac{1 - r}{r} \varphi_2 - \dfrac{1}{r} \varphi_1 + v_t + (1 - s_1) U_t - s\right)^{\frac{n + 1}{n}}\right) \omega^n
\\
&\leq 
\int_{M_s} \exp \left( \dfrac{ \alpha_0}{1 - s_1} \left(\dfrac{n + 1}{n}   \left(-\psi_{s,k} + (1 - s_1) (U_t + 1)\right) + A_{s,k }\right)\right) \omega^n
\\
&\leq 
\int_{M_s} \exp \left( \dfrac{ \alpha_0}{1 - s_1} \left(\dfrac{n + 1}{n}   \left(-\psi_{s,k} + (1 - s_1) \right) + A_{s,k }\right)\right) \omega^n
\\
&\leq C \exp \left(\dfrac{\alpha_0}{1 - s_1} A_{s,k}\right)
.
\end{aligned}
\end{equation}
Letting $k$ approach infinity in the above inequality~\eqref{inequality-4-29}, we have
\begin{equation}
\label{inequality-4-30}
\begin{aligned}
\int_{M_s} \exp \left( \dfrac{\alpha_0}{(1 - s_1)A^{\frac{1}{n}}_s } \left(\dfrac{1 - r}{r} \varphi_2 - \dfrac{1}{r} \varphi_1 + v_t + (1 - s_1) U_t - s\right)^{\frac{n + 1}{n}}\right) \omega^n
\\
\leq 
C \exp \left(\dfrac{\alpha_0}{1 - s_1} A_s\right)
,
\end{aligned}
\end{equation}
where
\begin{equation}
A_s
:=
\dfrac{1}{(1 - s_1)^n V_t} \int_{M_s} \left(\dfrac{1 - r}{r} \varphi_2 - \dfrac{1}{r} \varphi_1 + v_t + (1 - s_1) U_t - s\right)  \dfrac{c_t f_1 + \sigma_1}{r^n} \omega^n
.
\end{equation}
On $M_s$ with $s \geq K_2$, 
\begin{equation}
\dfrac{1 - r}{r} (\varphi_2 - \varphi_1) 
\geq  
\varphi_1 - v_t - (1 - s_1) U_t + s
\geq 
0
,
\end{equation}
and consequently
\begin{equation}
\label{inequality-4-33}
\begin{aligned}
A_s
&\leq 
\dfrac{1}{(1 - s_1)^n V_t} \int_{M_s}  \dfrac{1 - r}{r} (\varphi_2 -   \varphi_1)   \dfrac{c_t f_1 + \sigma_1}{r^n} \omega^n
\\   
&\leq 
\dfrac{1}{(1 - s_1)^n V_t r^{n + 1}}  \Vert (\varphi_2 - \varphi_1)^+ \Vert_{L^{q^*}} \Vert c_t f_1 + \sigma_1 \Vert_{L^q}   
. 
\end{aligned}
\end{equation}
The last line of \eqref{inequality-4-33} is due to H\"older inequality with respect to measure $\omega^n$.

Supposing that $\Vert (\varphi_2 - \varphi_1)^+ \Vert_{L^{q^*}} < \dfrac{1}{2^{n + 2}} $ with $q^* := \dfrac{q}{q - 1}$, we may choose
\begin{equation}
r := \Vert (\varphi_2 - \varphi_1)^+ \Vert^{\frac{1}{n + 2}}_{L^{q^*}}  < \dfrac{1}{2} .
\end{equation}
Then we can get an upper bound of $A_s$, 
\begin{equation}
\label{inequality-4-35}
	A_s \leq E_t := \dfrac{1}{2 (1 - s_1)^n V_t} \Vert c_t f_1 + \sigma_1\Vert_{L^q}, 
\end{equation}
where $E_t$ is actually $t$-independently bounded if $t$ is bounded. 
For any 
$q > 1$, 
\begin{equation}
\label{inequality-4-36}
\begin{aligned}
	A_s
	&\leq 
	\dfrac{1}{(1 - s_1)^n V_t r^n} 
	\Vert c_t f_1 + \sigma_1 \Vert^{1 - \frac{1}{(n + 1)^2}}_{L^1 (M_s)}
	\Vert c_t f_1 + \sigma_1 \Vert^{\frac{1}{(n + 1)^2}}_{L^q (M_s)}
	\\
	&\qquad
	\cdot
	\left\Vert \left(\frac{1 - r}{r} \varphi_2 - \dfrac{1}{r} \varphi_1 + v_t + (1 - s_1) U_t - s \right)^{\frac{n + 1}{n}}\right\Vert^{\frac{n}{n + 1}}_{L^{\frac{q n (n + 1)}{q - 1}} (M_s)}
	\\
	&=
	\dfrac{A^{\frac{1}{n + 1}}_s}{\alpha^{\frac{n}{n + 1}} (1 - s_1)^n V_t r^n} 
	\Vert c_t f_1 + \sigma_1 \Vert^{1 - \frac{1}{(n + 1)^2}}_{L^1 (M_s)}
	\Vert c_t f_1 + \sigma_1 \Vert^{\frac{1}{(n + 1)^2}}_{L^q (M_s)}
	\\
	&\qquad 
	\cdot
	\left\Vert \dfrac{\alpha}{A^{\frac{1}{n}}_s} \left(\dfrac{1 - r}{r} \varphi_2 - \dfrac{1}{r} \varphi_1 + v_t + (1 - s_1) U_t - s \right)^{\frac{n + 1}{n}}\right\Vert^{\frac{n}{n + 1}}_{L^{\frac{q n (n + 1)}{q - 1}} (M_s)}
	,
\end{aligned}
\end{equation}
by H\"older inequality. 
Applying Taylor expansion and \eqref{inequality-4-30}\eqref{inequality-4-35} to \eqref{inequality-4-36}, 
\begin{equation}
\begin{aligned}
	A_s
	\leq 
	\dfrac{C 	\exp \left(\frac{\alpha_0 (q - 1)}{(1 - s_1) q n (n + 1)} E_t\right)}{\alpha_0 (1 - s_1)^n V^{\frac{n + 1}{n}}_t r^{n + 1}}
	\Vert c_t f_1 + \sigma_1 \Vert^{1 + \frac{1}{n + 1}}_{L^1 (M_s)}
	\Vert c_t f_1 + \sigma_1 \Vert^{\frac{1}{n (n + 1)}}_{L^q (M_s)}
	.
\end{aligned}
\end{equation}
For any $s' > 0$ and $s \geq K_2$,
\begin{equation}
\label{inequality-4-38}
\begin{aligned}
&\quad
s' \int_{M_{s + s'}} (c_t f_1 + \sigma_1) \omega^n
\\
&
\leq 
\int_{M_s} \left(\dfrac{1 - r}{r} \varphi_2 - \dfrac{1}{r} \varphi_1 + v_t + (1 - s_1) U_t - s\right) (c_t f_1 + \sigma_1) \omega^n 
\\
&
\leq 
\dfrac{C 	\exp \left(\dfrac{\alpha_0 (q - 1)}{(1 - s_1) q n (n + 1)} E_t\right)}{\alpha_0  V^{\frac{1}{n}}_t r }
\Vert c_t f_1 + \sigma_1 \Vert^{\frac{1}{n (n + 1)}}_{L^q}
\left(\int_{M_s} (c_t f_1 + \sigma_1) \omega^n\right)^{1 + \frac{1}{n + 1}}
.
\end{aligned}
\end{equation}
As a direct application, we have
\begin{equation}
\begin{aligned}
	\int_{M_{K_1 + 1}} (c_t f_1 + \sigma_1) \omega^n
	&\leq
	\dfrac{1 - r}{r} 
	\int_{M_{K_1}} (\varphi_2 - \varphi_1) (c_t f_1 + \sigma_1) \omega^n 
	\\
	&\leq 
	\dfrac{1}{r} 
	\int_{M_{K_1}} (\varphi_2 - \varphi_1) (c_t f_1 + \sigma_1) \omega^n 
	\\
	&\leq 
	r^{n + 1} \Vert c_t f_1 + \sigma_1\Vert_{L^q}
	.
\end{aligned}
\end{equation}
By De Giorgi iteration,
\begin{equation}
\begin{aligned}
&\quad 
\dfrac{1 - r}{r} \varphi_2 - \dfrac{1}{r} \varphi_1 + v_t + (1 - s_1) U_t
\\
&\leq 
K_2 + 1 
+
\dfrac{C 	\exp \left(\frac{\alpha_0 (q - 1)}{(1 - s_1) q n (n + 1)} E_t\right)}{\alpha_0 V^{\frac{1}{n}}_t }
\Vert c_t f_1 + \sigma_1 \Vert^{\frac{1}{n }}_{L^q}   
,
\end{aligned}
\end{equation}
and consequently
\begin{equation}
\label{inequality-4-3-58}
\begin{aligned}
&\quad
	\varphi_2 - \varphi_1
	\\
	&\leq 
	2 r 
	\left(
	\varphi_1 - v_t - (1 - s_1) U_t
	+
	K_2 + 1 
	+
	\dfrac{C 	\exp \left(\frac{\alpha_0 (q - 1)}{(1 - s_1)q n (n + 1)} E_t\right)}{\alpha_0 V^{\frac{1}{n}}_t }
	\Vert c_t f_1 + \sigma_1 \Vert^{\frac{1}{n }}_{L^q} \right)
	\\
	&\leq 
	2 r 
	\left(
	 - v_t + s_1 U_t -  U_t
	+
	K_2 + 1 
	+
	\dfrac{C 	\exp \left(\frac{\alpha_0 (q - 1)}{(1 - s_1) q n (n + 1)} E_t\right)}{\alpha_0 V^{\dfrac{1}{n}}_t }
	\Vert c_t f_1 + \sigma_1 \Vert^{\frac{1}{n }}_{L^q} \right)
	\\
	&\leq
	2 r 
	\left(
	 -  U_t
	+
	K_1 + K_2 + 1 
	+
	\dfrac{C 	\exp \left(\frac{\alpha_0 (q - 1)}{(1 - s_1) q n (n + 1)} E_t\right)}{\alpha_0  V^{\frac{1}{n}}_t }
	\Vert c_t f_1 + \sigma_1 \Vert^{\frac{1}{n }}_{L^q} \right)
	.
\end{aligned}
\end{equation}

If  $\Vert  (\varphi_2 - \varphi_1)^+\Vert_{L^{q^*}} \geq \frac{1}{2^{n + 2}}$, then
\begin{equation}
\label{inequality-4-3-59}
\begin{aligned}
	\varphi_2 - \varphi_1 
	&\leq 
	2 \left( - U_t +  \Vert  (- \varphi_1 + U_t)^+ \Vert_{L^\infty} \right) r
	.
\end{aligned}
\end{equation}

Combing \eqref{inequality-4-3-58} and \eqref{inequality-4-3-59}, we can conclude that
\begin{equation}
\label{inequality-4-3-61}
\varphi_2 - \varphi_1 
\leq 
2 ( - U_t + C) \Vert (\varphi_2 - \varphi_1)^+ \Vert^{\frac{1}{n + 2}}_{L^{q^*}}
\leq 
2 ( - U   + C) \Vert (\varphi_2 - \varphi_1)^+ \Vert^{\frac{1}{n + 2}}_{L^{q^*}}
.
\end{equation}

\begin{theorem}
Let $\varphi_1$ and $\varphi_2$ be solutions to Equation~\eqref{equality-4-1} and \eqref{equality-4-2} respectively. For $q > 1$, we have
\begin{equation}
	\sup_M (\varphi_2 - \varphi_1) 
	\leq 
	2 \left( - U_t + C (\chi,\tilde \chi, \omega, n , q, \Vert f_1\Vert_{L^q}) \right) \Vert (\varphi_2 - \varphi_1)^+ \Vert^{\frac{1}{n + 2}}_{L^{q^*}} ,
\end{equation}
where $q^* = \frac{q}{q - 1}$. 
\end{theorem}


\medskip
\section{Limiting function}
\label{limiting-function}

In this section, we shall study the limit function of solution $\varphi_t$ to Equation~\eqref{equation-1-5}. 
In case that $U$ is bounded, we actually construct a weak solution in pluripotential sense.

Since $U$ is $\tilde \chi$-PSH, we know that $U \in L^p $ for any $p \geq 1$. 
By the $L^\infty$ estimate in Section~\ref{L-estimate} and the fact that $U \leq U_t \leq 0$, we can see that $\{\varphi_t \}_{0 < t \leq 1}$ is uniformly bounded in $L^p$ for any $p \geq 1$.

The following gradient estimate is essentially in \cite{GuedjZeriahi2017}, and we include the details for completeness of the argument. 
We know from \eqref{inequality-3-1-19} that 
\begin{equation}
(L \omega + \sqrt{-1} \partial\bar\partial\varphi_t ) \wedge \omega^{n - 1} > 0 ,
\end{equation}
for some sufficiently large constant $L > 0$.
We define
\begin{equation}
	\varphi_{t, \sigma} := - ( - \varphi_t + 1)^\sigma \leq - 1 ,
\end{equation}
for $\sigma \in (0,1)$. 
By direct calculation, 
\begin{equation}
\begin{aligned}
&\quad 
	\left(L \omega + \sqrt{-1} \partial\bar\partial \varphi_{t, \sigma}\right) \wedge \omega^{n - 1}
	\\
	&\geq 
	L \left(1 - \sigma ( - \varphi_t + 1)^{\sigma - 1} \right) \omega^n 
	+
	\sigma (- \varphi_t + 1)^{\sigma - 1}\left(L \omega + \sqrt{- 1} \partial\bar\partial 
		\varphi_t\right)  \wedge \omega^{n - 1}
	\\
	&\qquad
	+ \sigma (1 - \sigma ) (- \varphi_t + 1)^{\sigma - 2} \sqrt{-1} \partial \varphi_t \wedge \bar\partial \varphi_t
		\wedge \omega^{n - 1}
	\\
	&>
	\sigma (1 - \sigma ) (- \varphi_t + 1)^{\sigma - 2} \sqrt{-1} \partial \varphi_t \wedge \bar\partial \varphi_t
	\wedge \omega^{n - 1}
	,
\end{aligned}
\end{equation}
and hence
\begin{equation}
\begin{aligned}
	\int_M \dfrac{ |\nabla \varphi_t|^2 }{(- \varphi_t + 1)^{  2 - \sigma}} \omega^n
	&\leq 
	\dfrac{n}{\sigma (1 - \sigma ) } \int_M 	\left(L \omega + \sqrt{-1} \partial\bar\partial \varphi_{t, \sigma}\right) \wedge \omega^{n - 1}
	\\
	&= 
	\dfrac{n L}{\sigma (1 - \sigma ) } \int_M \omega^n
	.
\end{aligned}
\end{equation}
By H\"older inequality, we obtain
\begin{equation}
\label{inequality-5-5}
\begin{aligned}
	\Vert \nabla \varphi_t \Vert_{L^r}
	&= 
	\left(\int_M \left(\dfrac{|\nabla \varphi_t|^2}{(- \varphi_t + 1)^{2 - \sigma}}\right)^{\frac{r}{2}}     ( - \varphi_t + 1)^{\frac{(2 - \sigma) r}{2}}\omega^n\right)^{\frac{1}{r}}
	\\
	&\leq 
	\left(\int_M \frac{|\nabla \varphi_t|^2}{(- \varphi_t + 1)^{2 - \sigma}}\right)^{\frac{1}{2}} 
	\left( \int_M (- \varphi_t + 1)^{\frac{(2 - \sigma) r}{2 - r}} \omega^n\right)^{  \frac{2 - r}{2 r}}
	,
\end{aligned}
\end{equation}
for $1 \leq r < 2$.

By the gradient estimate~\eqref{inequality-5-5}, $\{ \varphi_t\}_{0 < t \leq 1}$ is uniformly bounded in $W^{1,r}$ for any $1 \leq  r < 2$.
By Sobolev embedding,  $\{ \varphi_t\}_{0 < t \leq 1}$ is precompact in $L^1$ norm, and thus  there exists  a sequence $t_i$ decreasing to $0$ such that $\varphi_{t_i}$ is convergent in $L^1$ norm.
%
%
%
Therefore for any fixed $1 \leq q^* < +\infty$, 
\begin{equation}
\begin{aligned}
	\int_M |\varphi_{t_i} - \varphi_{t_j}|^{q^*} \omega^n
	&=
	\int_M |\varphi_{t_i} - \varphi_{t_j}|^{q^* - \frac{1}{2}} |\varphi_{t_i} - \varphi_{t_j}|^{\frac{1}{2}} \omega^n
	\\
	&\leq 
	\left( \int_M |\varphi_{t_i} - \varphi_{t_j}|^{2q^* - 1} \omega^n \right)^{\frac{1}{2}} 
	\left(\int_M |\varphi_{t_i} - \varphi_{t_j}| \omega^n\right)^{\frac{1}{2}}
	\\
	&\leq 
	\left(\Vert \varphi_{t_i}\Vert_{L^{2q^* - 1}} + \Vert \varphi_{t_j}\Vert_{L^{2q^* - 1}}\right)^{\frac{2 q^* - 1}{2}} 
	\left(\int_M |\varphi_{t_i} - \varphi_{t_j}| \omega^n\right)^{\frac{1}{2}}
	\\
	&\leq 
	C 	\left(\int_M |\varphi_{t_i} - \varphi_{t_j}| \omega^n\right)^{\frac{1}{2}}
	\\
	&\to 0 \qquad \qquad (i, j \to \infty) ,
\end{aligned}
\end{equation}
which means $\{\varphi_{t_i}\}$ is Cauchy in $L^{q^*}$. 
By passing to a subsequence again, we may assume that
\begin{equation}
	\Vert  \varphi_{t_j} - \varphi_{t_i}\Vert_{L^{q^*}} < \dfrac{1}{2^{(n + 2) (i + 2)}} , \qquad \forall j \geq i .
\end{equation}
By the stability estimate~\eqref{inequality-4-3-61}, 
\begin{equation}
	\varphi_{t_{i + 1}} - \varphi_{t_i} 
	\leq 
	2 (- U + C) \Vert \varphi_{t_{i + 1}} - \varphi_{t_i}\Vert^{\frac{1}{n + 2}}_{L^{q^*}} 
	\leq 
	2 (- U + C) \Vert \varphi_{t_{i + 1}} - \varphi_{t_i}\Vert^{\frac{1}{n + 2}}_{L^{q^*}} 
	\leq
	\dfrac{ - U + C}{2^{i + 1}}
	.
\end{equation}
For fixed $j$ and $\forall i > j$,
\begin{equation}
\left(1 + \dfrac{1}{2^j}\right) (U - C)
\leq 
\varphi_{t_{i + 1}} + \left(\dfrac{1}{2^j} - \dfrac{1}{2^{i + 1}}\right) (U - C)
\leq
\varphi_{t_i} + \left(\dfrac{1}{2^j} - \dfrac{1}{2^i}\right) (U - C)
\leq 
0
.
\end{equation}
We can conclude that $\varphi_{t_i} + \left(\frac{1}{2^j} - \frac{1}{2^i}\right) (U - C)$ is convergent to a
function $\varphi + \frac{1}{2^j} (U - C)$,
which is $\left(\chi + \tilde \chi - \cot (\theta_0) \omega + \frac{1}{2^j} \tilde \chi\right)$-PSH.
Then $\varphi^*$ is a $\left(\chi + \tilde \chi - \cot (\theta_0) \omega \right)$-PSH function.
In fact, $\varphi^* + U$ is $\left(\chi + 2\tilde \chi - \cot (\theta_0) \omega \right)$-PSH, 
and $\varphi^* + U = \varphi + U$ almost everywhere.
Therefore, $\varphi^* + U = \varphi + U$, that is $\varphi^* = \varphi$.

Moreovr, the limit function is a weak solution in pluripotential sense 
if $U$ is bounded. 

\medskip
\section{Stability estimate for dimensional $3$}

When $n = 3$, Equation~\eqref{equation-4-3} is not known to be solvable under the requirement~\eqref{4-5}. So we have to adapt  slightly different arguments.

\medskip

\subsection{Case 1: $\theta_0 \in (0,{\pi}/{2}]$, i.e. hypercritical phase case}

In this case, $\cot (\theta_0) \geq 0$. We plan to rewrite the equation as a Complex Monge-Amp\`ere type equation, and then adapt the argument in~\cite{Sun202210}. 

We rewrite Equation~\eqref{equation-1-5} as
\begin{equation}
\label{equation-4-15}
\begin{aligned}
&
	(\chi + \tilde \chi + t\omega + \sqrt{-1} \partial\bar\partial \varphi_t)^3 - 3 (\chi + \tilde \chi + t\omega + \sqrt{-1} \partial\bar\partial \varphi_t) \wedge \omega^2 \\
&=
\cot(\theta_0)
\left(
	3 (\chi + \tilde \chi + t\omega + \sqrt{-1} \partial\bar\partial \varphi_t)^2 \wedge \omega
	-
	\omega^3
\right)
	+ 
	c_t f \omega^3
	.
\end{aligned}
\end{equation}
Given that $\bm{\lambda} (\chi + \tilde \chi + t\omega + \sqrt{-1} \partial\bar\partial \varphi_t) \in \Gamma_{\theta_0,\Theta_0}$, we know that 
\begin{equation}
\chi + \tilde \chi + t\omega + \sqrt{-1} \partial\bar\partial \varphi_t > \cot (\theta_0) \omega,
\end{equation}
and hence rewrite Equation~\eqref{equation-4-15} as
\begin{equation}
\label{equation-5-3}
\begin{aligned}
&
	(\chi + \tilde \chi + t\omega - \cot (\theta_0) \omega+ \sqrt{-1} \partial\bar\partial \varphi_t )^3 \\
&=
	3 \left(\cot^2 (\theta_0) + 1\right)(\chi + \tilde \chi + t\omega - \cot (\theta_0) \omega + \sqrt{-1} \partial\bar\partial \varphi_t  ) \wedge \omega^2 
	\\
&\qquad
	+
	2 \cot(\theta_0) \left(\cot^2(\theta_0) + 1\right) \omega^3
	+ 
	c_t f \omega^3
	,
\end{aligned}
\end{equation}
by a simple polynomial expansion. 
Then \eqref{equality-4-1} and \eqref{equality-4-2} can be rewritten as
\begin{equation}
\label{equation-7-4}
\begin{aligned}
&
	(\chi + \tilde \chi + t\omega - \cot (\theta_0) \omega+ \sqrt{-1} \partial\bar\partial \varphi_1 )^3 \\
&=
	3 \left(\cot^2 (\theta_0) + 1\right)(\chi + \tilde \chi + t\omega - \cot (\theta_0) \omega + \sqrt{-1} \partial\bar\partial \varphi_1  ) \wedge \omega^2 
	\\
&\qquad
	+
	2 \cot(\theta_0) \left(\cot^2(\theta_0) + 1\right) \omega^3
	+ 
	c_t f_1 \omega^3
	,
	\qquad
	\sup_M \varphi_1 = 0 ,
\end{aligned}
\end{equation}
and
\begin{equation}
\label{equation-7-5}
\begin{aligned}
&
	(\chi + \tilde \chi + t\omega - \cot (\theta_0) \omega+ \sqrt{-1} \partial\bar\partial \varphi_2 )^3 \\
&=
	3 \left(\cot^2 (\theta_0) + 1\right)(\chi + \tilde \chi + t\omega - \cot (\theta_0) \omega + \sqrt{-1} \partial\bar\partial \varphi_2  ) \wedge \omega^2 
	\\
&\qquad
	+
	2 \cot(\theta_0) \left(\cot^2(\theta_0) + 1\right) \omega^3
	+ 
	c_t f_2 \omega^3
	,
	\qquad
	\sup_M \varphi_2 = 0 .
\end{aligned}
\end{equation}
Moreover, the boundary case of $\mathcal{C}$-subsolution condition can be read as
\begin{equation}
	\left(\chi - \cot (\theta_0) \omega\right)^2 \geq \left(1 + \cot^2 (\theta_0)\right) \omega^2 ,
\end{equation}
and
\begin{equation}
	\chi -  \cot (\theta_0) \omega > 0 .
\end{equation}
Since $\cot (\theta) \geq 0$ when $\theta \in \left(0,\dfrac{\pi}{2}\right]$, we derive as in \cite{Sun202210} that
\begin{equation}
\begin{aligned}
&\quad
	\dfrac{\left( 2 \cot (\theta_0) \left(\cot^2 (\theta_0) + 1\right) + c_t f_1 \right) \omega^3}{r^3 (\chi - \cot(\theta_0) \omega + \hat\chi)^3} - 1
	\\
	&\geq
	\dfrac{ \left( 2 \cot (\theta_0) \left(\cot^2 (\theta_0) + 1\right) + c_t f_1 \right) \omega^3}{ (\chi + \tilde \chi + t\omega - \cot (\theta_0 ) \omega + \sqrt{-1} \partial\bar\partial\varphi_1)^3} - 1
	\\
	&\geq 
	- 3 (1 - r)  \left( \cot^2 (\theta_0) + 1\right) \dfrac{ (\chi + \tilde \chi + t\omega - \cot (\theta_0 ) \omega + \sqrt{-1} \partial\bar\partial\varphi_2) \wedge \omega^2}{ (\chi + \tilde \chi + t\omega - \cot (\theta_0 ) \omega + \sqrt{-1} \partial\bar\partial\varphi_2)^3}
	\\
	&\qquad 
	- 3 r \left( \cot^2 (\theta_0) + 1\right)  \dfrac{(\chi - \cot (\theta_0) \omega + \hat\chi) \wedge \omega^2}{(\chi - \cot (\theta_0) \omega + \hat\chi)^3}
	\\
	&\geq 
	- (1 - r) 
	- 3 r \left( \cot^2 (\theta_0) + 1\right)  \dfrac{(\chi - \cot (\theta_0) \omega + \hat\chi) \wedge \omega^2}{(\chi - \cot (\theta_0) \omega + \hat\chi)^3}
	,
\end{aligned}
\end{equation}
where $\hat\chi := \tilde \chi + t \omega + \dfrac{1}{r} \sqrt{-1} \partial\bar\partial \varphi_1 - \dfrac{1 - r}{r} \sqrt{-1} \partial\bar\partial \varphi_2 \geq 0$. 
Then
\begin{equation}
	\dfrac{  c_t f_1  +  C (\chi,\omega,\theta_0) }{r^4} \omega^3 
	\geq 
	\hat\chi^3
	.
\end{equation}

Similar to the arguments in Section~\ref{stability} and Section~\ref{limiting-function}, we can derive the following results. First, we have a stability estimate.
\begin{theorem}
Let $\varphi_1$ and $\varphi_2$ be solutions to Equation~\eqref{equation-7-4} and \eqref{equation-7-5} respectively. For $q > 1$, we have
\begin{equation}
	\sup_M (\varphi_2 - \varphi_1) 
	\leq 
	2 \left( - U_t + C (\chi,\tilde \chi, \omega, q, \Vert f_1\Vert_{L^q}) \right) \Vert (\varphi_2 - \varphi_1)^+ \Vert^{\frac{1}{6}}_{L^{q^*}} ,
\end{equation}
where $q^* = \frac{q}{q - 1}$.

\end{theorem}

%
%

Second, a limit function can be constructed through the stability estimate.
%
While $U$ is bounded, the limit function is indeed a weak solution in pluripotential sense. 

\medskip

\subsection{Case 2: $\theta_0 \in (\pi/2,\pi)$, i.e. supercritical but not hypercritical phase case}
In this case, $\cot (\theta_0) < 0$. 
We shall utilize the following concavity property~\cite{HuiskenSinestrari}.
\begin{theorem}
The cone $\Gamma^k \subset \mathbb{R}^n$ is convex and the function $\frac{S_{k + 1}}{S_k}$ is concave on $\Gamma^k$ for  $1 \leq k \leq  n -1$.
\end{theorem}
The author was told about the theorem in a meeting with Bo Guan and Rirong Yuan. 
To apply the theorem, we impose an extra assumption that $\bm{\lambda} (\chi) \in \bar\Gamma^2$ in the following argument. It is easy to see that for all $ t > 0$, $\bm{\lambda} (\chi + t\omega) \in \Gamma^2$.  

We shall do a decomposition as follows,
\begin{equation}
\begin{aligned}
&\quad
\chi + \tilde \chi + t\omega + \sqrt{-1} \partial\bar\partial \varphi_1
\\
&=
(1 - r) \left(\chi + \tilde \chi + t \omega  + \sqrt{-1} \partial\bar\partial \varphi_2\right) 
+ r \left( \chi + \frac{t}{2} \omega + \hat\chi  \right)
,
\end{aligned}
\end{equation}
where $0 < r < 1$ and
\begin{equation}
\hat \chi := \tilde \chi + \frac{t}{2} \omega  + \frac{1}{r} \sqrt{-1} \partial\bar\partial \varphi_1 - \frac{1 - r}{r} \sqrt{-1} \partial\bar\partial \varphi_2 .
\end{equation}
Wherever $\hat\chi \geq 0$, we can derive that
\begin{equation}
\begin{aligned}
&\quad
3 \cot (\theta_0) + \left(c_t f_1 - \cot (\theta_0)\right) \frac{\omega^3}{r^2 \left(\chi + \frac{t}{2} \omega + \hat\chi\right)^2 \wedge \omega}
\\
&\geq
3 \cot (\theta_0) + \left(c_t f_1 - \cot (\theta_0)\right) \frac{\omega^3}{(\chi + \tilde \chi + t\omega + \sqrt{-1} \partial\bar\partial \varphi_1)^2 \wedge \omega}
\\
&\geq
(1 - r) \frac{(\chi + \tilde \chi + t\omega + \sqrt{-1} \partial\bar\partial \varphi_2)^3}{(\chi + \tilde \chi + t\omega + \sqrt{-1} \partial\bar\partial \varphi_2)^2 \wedge \omega} 
- 
3 (1 - r) \frac{(\chi + \tilde \chi + t\omega + \sqrt{-1} \partial\bar\partial \varphi_2) \wedge \omega^2}{(\chi + \tilde \chi + t\omega + \sqrt{-1} \partial\bar\partial \varphi_2)^2 \wedge \omega}
\\
&\qquad
+
r \frac{\left(\chi + \frac{t}{2} \omega + \hat\chi\right)^3}{\left(\chi + \frac{t}{2} \omega + \hat\chi\right)^2 \wedge \omega } - 3 r \frac{\left(\chi + \frac{t}{2} \omega + \hat\chi\right) \wedge \omega^2}{\left(\chi + \frac{t}{2} \omega + \hat\chi\right)^2 \wedge \omega}
\\
&\geq 
(1 - r)  3 \cot (\theta_0) 
+
r \frac{\left(\chi + \frac{t}{2} \omega + \hat\chi\right)^3}{\left(\chi + \frac{t}{2} \omega + \hat\chi\right)^2 \wedge \omega } - 3 r \frac{\left(\chi + \frac{t}{2} \omega + \hat\chi\right) \wedge \omega^2}{\left(\chi + \frac{t}{2} \omega + \hat\chi\right)^2 \wedge \omega}
,
\end{aligned}
\end{equation}
and hence
\begin{equation}
\begin{aligned}
&\quad
	3 \cot (\theta_0) \left(\chi + \frac{t}{2} \omega + \hat\chi\right)^2 \wedge \omega + \frac{c_t f_1 - \cot (\theta_0)}{r^3} \omega^3
	\\
	&\geq 
	\left(\chi + \frac{t}{2} \omega + \hat\chi\right)^3 - 3  \left(\chi + \frac{t}{2} \omega + \hat\chi\right) \wedge \omega^2  .
\end{aligned}
\end{equation}
By the boundary case condition,
\begin{equation}
\begin{aligned}
&\quad
\frac{c_t f_1 - \cot (\theta_0)}{r^3}  \omega^3 + \cot (\theta_0) \omega^3
\\
	&\geq \mathfrak{Re} \left(\chi + \frac{t}{2} \omega + \hat\chi + \sqrt{-1} \omega\right)^3 - \cot(\theta_0) \mathfrak{Im} \left(\chi + \frac{t}{2} \omega + \hat\chi + \sqrt{-1} \omega\right)^3 
	\\
	&\geq 
	\left(\hat\chi + \frac{t}{2}\omega\right)^3 + \chi^3 - 3 \chi \wedge \omega^2 
	- 3 \cot (\theta_0) \chi^2 \wedge \omega + \cot (\theta_0) \omega^3
	.
\end{aligned}
\end{equation}
Therefore,
\begin{equation}
\frac{c_t f_1 - \cot (\theta_0) + C(\chi,\omega)}{r^3}
	\geq
	 \frac{c_t f_1 - \cot (\theta_0)}{r^3}   \omega^3 + 3 \chi \wedge \omega^2 
	 \geq 
	\hat\chi^3 
	.
\end{equation}

Similar to the arguments in Section~\ref{stability} and Section~\ref{limiting-function}, we can derive the following results. First, we have a stability estimate.
\begin{theorem}
Let $\varphi_1$ and $\varphi_2$ be solutions to Equation~\eqref{equation-7-4} and \eqref{equation-7-5} respectively. For $q > 1$, we have
\begin{equation}
	\sup_M (\varphi_2 - \varphi_1) 
	\leq 
	2 \left( - U_t + C (\chi,\tilde \chi, \omega, q, \Vert f_1\Vert_{L^q}) \right) \Vert (\varphi_2 - \varphi_1)^+ \Vert^{\frac{1}{5}}_{L^{q^*}} ,
\end{equation}
where $q^* = \frac{q}{q - 1}$.

\end{theorem}

Second, a limit function can be constructed through the stability estimate.
%
While $U$ is bounded, the limit function is indeed a weak solution in pluripotential sense. 

\medskip

\noindent
{\bf Acknowledgements}\quad
The author wish to thank  
Gao Chen, 
Bo Guan 
and
Rirong Yuan
for helpful discussions.  
The author is supported by 
National Natural Science Foundation of China (12371207) 
and
a start-up grant from ShanghaiTech University (2018F0303-000-04).

\end{document}